\documentclass[11pt]{amsart}
\usepackage[margin=1in]{geometry}                
\usepackage[parfill]{parskip}    
\usepackage{graphicx}
\usepackage{amssymb}
\usepackage{epstopdf}
\usepackage{latexsym,enumitem,amsmath,amsthm,amsfonts,bbm,hyperref,amsrefs,bm,stmaryrd,tikz-cd}
\usepackage{comment}

\newtheorem{theorem}{Theorem}
\newtheorem{corollary}[theorem]{Corollary}
\newtheorem{lemma}[theorem]{Lemma}
\newtheorem{proposition}[theorem]{Proposition}
\newtheorem*{lem:kirkgoebel}{Lemma \ref{lem:kirkgoebel}}

\theoremstyle{definition}
\newtheorem{definition}[theorem]{Definition}
\newtheorem{remark}[theorem]{Remark}

\newtheorem{example}[theorem]{Example}

\makeatletter
\newcommand{\dotminus}{\mathbin{\text{\@dotminus}}}

\newcommand{\@dotminus}{%
  \ooalign{\hidewidth\raise1ex\hbox{.}\hidewidth\cr$\m@th-$\cr}%
}
\makeatother

\newcommand{\Sub}{\text{Sub} \,}
\newcommand{\Subm}{\text{Sub}_{\text{m}} \,}
\newcommand{\Sube}{\text{Sub}_{\epsilon} \,}
\newcommand{\Subeo}{\text{Sub}_{\epsilon_1} \,}
\newcommand{\Subet}{\text{Sub}_{\epsilon_2} \,}
\newcommand{\Subeth}{\text{Sub}_{\epsilon_3} \,}

\newcommand{\Subeef}{\text{Sub}_{(\epsilon \circ \epsilon_f)} \,}
\newcommand{\Subi}{\text{Sub}_{\mathbb{I}} \,}
\newcommand{\Suboi}{\text{Sub}_{1_{\mathbb{I}}} \,}

\newcommand{\TSub}{\text{Sub}_{\mathbbm{2}} \,}

\newcommand{\Twop}{\mathbbm{2}^{\text{op}}}
\newcommand{\Vop}{\mathcal{V}^{\text{op}}}

\newcommand{\Iop}{\mathbb{I}^{\text{op}}}
\newcommand{\Cop}{\mathcal{C}^{\text{op}}}
\newcommand{\pmetu}{\mathbf{pMet_{1}^u}}
\newcommand{\pmetl}{\mathbf{pMet_{1}^L}}
\newcommand{\pmet}{\mathbf{pMet_{1}^{1}}}
\newcommand{\pmetv}{\mathbf{pMet}}
\newcommand{\catsets}{\mathbf{Set}}

\newcommand{\Psh}{\text{PSh}}

\newcommand{\mcc}{\mathcal{C}}

\newcommand{\mcv}{\mathcal{V}}
\newcommand{\mct}{\mathcal{T}}
\newcommand{\mcti}{\mathcal{T}_{\mathbb{I}}}

\newcommand{\mcs}{\mathcal{S}}

\newcommand{\Endo}{\text{End}}

\title[Categorical semantics of metric spaces and continuous logic]{Categorical semantics of metric spaces \\ and continuous logic}
\author{Simon Cho}

\begin{document}
\maketitle

\begin{abstract}
Using the category of metric spaces as a template, we develop a metric analogue of the categorical semantics of classical/intuitionistic logic, and show that the natural notion of \emph{predicate} in this ``continuous semantics'' is equivalent to the a priori separate notion of predicate in \emph{continuous logic}, a logic which is independently well-studied by model theorists and which finds various applications. We show this equivalence by exhibiting the real interval $[0,1]$ in the category of metric spaces as a ``continuous subobject classifier'' giving a correspondence not only between the two notions of predicate, but also between the natural notion of quantification in the continuous semantics and the existing notion of quantification in continuous logic.

Along the way, we formulate what it means for a given category to behave like the category of metric spaces, and afterwards show that any such category supports the aforementioned continuous semantics. As an application, we show that categories of presheaves of metric spaces are examples of such, and in fact even possess continuous subobject classifiers.
\end{abstract}

\section{Introduction}

In categorical semantics of classical or intuitionistic logic, if the interpreting category is sufficiently nice then one has a \emph{subobject classifier}, which is an object $\Omega$ with a monomorphism $\top: 1 \hookrightarrow \Omega$ such that for every object $X$, there is a natural correspondence between maps $f: X \rightarrow \Omega$ and monomorphisms $i: A \hookrightarrow X$, obtained by pulling back $\top$ across $f$ as below:
\begin{equation*}
\begin{tikzcd}
A
\arrow{r}
\arrow{d}[swap]{i}
\arrow[phantom, "\lrcorner", very near start]{dr}
	& 1
	\arrow{d}{\top}\\
X
\arrow{r}{f}
	& \Omega
\end{tikzcd}
\end{equation*}
in particular yielding a well-behaved notion of characteristic function.

Having a subobject classifier is an important part of the categorical perspective on sets as the premier model of classical logic. Now the fact that sets are the ``best'' model of this logic is something that e.g. model theorists feel ``in their gut'' without having to refer to the categorical viewpoint; but this latter viewpoint has been valuable in identifying the underlying structure that allows categories such as (but not limited to) that of sets to interpret classical logic to various degrees, see e.g. \cite{johnstone}, \cite{moemac}, \cite{makrey}.

On the other hand, \emph{continuous logic} is a metric analogue of the usual logic; models of continuous logic are metric spaces, and the equality predicate is now replaced with a distance function. Strikingly, most model theoretic concepts have a well-behaved continuous analogue - see \cite{bbhu} for a detailed reference. Due to its richness, continuous logic is catalyzing an ongoing flurry of activity that is seeing to both the rapid development of the subject itself \cite{clogicdev1}, \cite{clogicdev2} and also applications to e.g. operator algebras \cite{operator1}, \cite{operator2} or proof mining \cite{avigadiovino}, \cite{cho}.

Just as with classical logic and the category of sets, we might ask for a category theoretic reason that metric spaces exhibit such rich logical structure. We examine the category of metric spaces with this question in mind; after some preliminaries in Section \ref{sec:prelim}, we identify relevant features of the category of metric spaces and use them to develop a ``categorical semantics of metric spaces'' in Section \ref{sec:indexedsubs}. One fruit of this approach is that we are able to identify a natural and general notion of ``continuous subobject'', and show that the real (unit) interval in the category of metric spaces - the object of truth values of continuous logic - is a ``continuous subobject classifier'', in the sense that maps into the real interval naturally correspond via pullback to continuous subobjects.

Now in continuous logic, predicates on a space $X$ are just maps from $X$ into the real interval, i.e. they are continuous real-valued features of (points in) the space $X$. The result of Section \ref{sec:indexedsubs} shows that using continuous subobjects of $X$ works equally well; indeed, maps into the real interval are exactly just the characteristic functions of these continuous subobjects. One practical advantage of the continuous subobjects approach is that we are able to make sense of this idea of ``continuous real-valued feature'' in other categories where there might not be an obvious ``object of real truth values''. In particular, we are able to apply the machinery we develop in Section \ref{sec:indexedsubs} \emph{without modification} in Section \ref{sec:presheaves} to the category of \emph{presheaves} of metric spaces, automatically making sense of the notion of ``metric'' and ``continuous subobject'' in this more general setting. Using this we show that the category of presheaves of metric spaces possesses a continuous subobject classifier, and therefore has a well-behaved notion of characteristic function for continuous real-valued features of presheaves of metric spaces.

The author is indebted to Andreas Blass for helpful discussions and suggestions.

\section{Preliminaries}
\label{sec:prelim}

\subsection{The basic model}

We work in the category $\pmetu$ of diameter $\leq 1$ pseudometric spaces (metric spaces whose distances can be $0$ for distinct points) and uniformly continuous maps. A comprehensive reference for such spaces is \cite{cech}. We allow these slightly more general spaces for technical convenience, and we work with uniformly continuous maps because they are the maps in continuous logic. (We will simply call these metric spaces.)

Essentially all that follows works equally well for the categories $\pmetl$ (resp. $\pmet$), which have the same objects as $\pmetu$ but with only Lipschitz (resp. $1$-Lipschitz) maps between them. Moreover, we may even remove the requirement that the diameter of our spaces be bounded, as long as we take the appropriate corresponding quantale to be $(\mathbb{R}_{\geq 0}, +, 0)$ instead of $([0,1], +, 0)$ (refer to Section \ref{sec:indexedsubs} of this paper).

Keeping in mind this flexibility in our choice of variants of metric spaces, let us simply write $\pmetv$ even though we are ostensibly working in $\pmetu$, to emphasize the fact that all of our analysis goes through essentially unchanged for the other variants. That is, strictly speaking we work with the choice $\pmetv = \pmetu$, but whenever $\pmetu$ is not explicitly mentioned then it is fine to choose $\pmetv = \pmetl$ or $\pmetv = \pmet$ instead. In the few places where this is not the case, we will say so explicitly.

\subsection{Categorical semantics}

We briefly recall some of the structure required of a category to support an interpretation of first order logic. A standard reference is e.g. \cite{johnstone}.

Given an object $X \in \mcc$, recall that a \emph{subobject} of $X$ is an appropriate equivalence class of monomorphisms into $X$. These objects form a poset category $\Subm X$ which can be seen as a skeleton of the full subcategory of $\mcc / X$ on the monomorphisms with codomain $X$.

Furthermore, if $\mcc$ has pullbacks, then (since pullbacks preserve monomorphisms and the equivalence relation mentioned above) any morphism $f: X \rightarrow Y$ induces via pullback a functor $f^\ast: \Subm Y \rightarrow \Subm X$. 

On the other hand, we might ask for a well-behaved notion of \emph{image} of morphisms $f: X \rightarrow Y$, so that any such morphism has a (necessarily unique, up to unique isomorphism) factorization $f = me$ where $m$ is the ``smallest'' mono through which $f$ factors, in the sense that for any other factorization $f = m^\prime e^\prime$ with $m^\prime$ monic, we have a commutative diagram
\begin{tikzcd}
X
\arrow{r}{e}
\arrow{d}[swap]{e^\prime}
	& \cdot
	\arrow[dotted]{dl}
	\arrow{d}{m} \\
\cdot
\arrow{r}[swap]{m^\prime}
	& Y
\end{tikzcd}
where the dotted diagonal is necessarily both monic and unique. We say that $\mcc$ \emph{has images} when every morphism $f$ has a factorization as above.

\begin{definition}
\label{def:regcat}
\item

\begin{enumerate}

\item A category $\mcc$ is called \emph{regular} when it has finite limits and images, and pulling back across morphisms preserves images.

\item A regular category $\mcc$ is called \emph{geometric} when for each $X \in \mcc$, $\Subm X$ is a small-complete lattice (i.e. has all small joins and meets) and pulling back across any $f: X \rightarrow Y$ preserves this structure.

\end{enumerate}

\end{definition}

In the definition above we can ignore issues of size (i.e. replace ``small-complete'' with ``complete'') if we assume that the objects of $\Subm X$ form an honest \emph{set}. This is referred to as \emph{wellpoweredness} in the literature, and in practice this is not an overly stringent condition.

The importance of the above structures lies in the fact that they conveniently describe sufficient conditions for a category to interpret first order logic, in which ``predicates of type $X$'' are interpreted as ``subobjects of $X$''.

The problem is that subobjects tend to be poorly behaved in topologically flavored categories like $\pmetv$; there are easy examples of $X \in \pmetv$ with many inequivalent subobjects which nevertheless should represent the same subspace. We may circumvent this problem by passing to a variation of the ``monomorphisms are subobjects'' viewpoint, by considering \emph{regular} monomorphisms as subobjects. The precise modifications required for this are well-known, but for completeness' sake we record them below.

\begin{definition}
\label{def:regmono}

A \emph{regular monomorphism} in a category $\mcc$ is a monomorphism which occurs as the equalizer of some parallel pair of morphisms. 

\end{definition}

In exceedingly nice categories such as $\catsets$, all monomorphisms are regular. The facts below support the idea that regular monomorphisms give the correct notion of subobject in $\pmetv$.

\begin{lemma}

\begin{enumerate}

\item The monomorphisms in $\pmetv$ are precisely the maps which are injective as functions.

\item The epimorphisms in $\pmetv$ are precisely the maps which are surjective as functions.

\end{enumerate}

\end{lemma}

\begin{proof}

That injective (resp. surjective) implies monic (resp. epi) is obvious. For the other directions, notice that if $m: A \rightarrow X$ fails to be injective, then maps from the one-point space into $A$ witness the failure of $m$ to be monic. Similarly, if $e: E \rightarrow B$ fails to be surjective, then maps from $B$ into the space of two points at distance $0$ witness the failure of $e$ to be epi.

\end{proof}

Let us call two monos $m_1: A_1 \rightarrow X$ and $m_2: A_2 \rightarrow X$ \emph{equivalent} when they are isomorphic as objects of the slice category $\pmetv / X$. In light of the characterization of regular monos in terms of slice categories, we see that if $m_1$ and $m_2$ are equivalent then $m_1$ is regular iff $m_2$ is.

\begin{proposition}
\label{prop:regmonoembeds}

A mono $m: A \rightarrow X$ in $\pmetv$ is regular if and only if it is equivalent to an isometric embedding.

\end{proposition}

\begin{proof}

To prove the left-to-right direction, we start by showing that for any parallel pair of maps
\begin{tikzcd}
X
\arrow[shift left]{r}{f}
\arrow[shift right]{r}[swap]{g}
		& Y
\end{tikzcd}
there is an isometric embedding $i: A \rightarrow X$ giving an equalizer to the pair $f, g$.

Let $A \subseteq X$ have the set of points $\{ x \in X \mid f(x) = g(x)\}$, and let the metric on $A$ be the restriction of the metric on $X$. Clearly the inclusion $i: A \rightarrow X$, which is an isometric embedding, is an equalizer for $f$ and $g$. Note that any other mono $m: A^\prime \rightarrow X$ occurring as an equalizer for $f,g$ must be equivalent to $i$.

Now given any regular mono $m$ which equalizes some parallel pair $f,g$, by the above we can find an isometric embedding $i$ also equalizing $f,g$, therefore showing that $m$ is equivalent to an isometric embedding.

For the right-to-left direction, let $i: A \rightarrow X$ be an isometric embedding, i.e. an inclusion $A \subseteq X$. Let $Y$ be the space with two points at distance $0$ from each other, and label them $a, b$. Define $f: X \rightarrow Y$ to be $f(x) = a$ for all $x \in X$, and define $g: X \rightarrow Y$ to be $g(x) = a$ iff $x \in A$, and $g(x) = b$ otherwise. Clearly $i$ must be an equalizer for $f, g$, so that any mono equivalent to $i$ must also be an equalizer for $f,g$.

\end{proof}

Henceforth, any time we speak of ``subobject'' we mean the equivalence classes defined above of \emph{regular} monomorphisms:

\begin{definition}
\label{def:strsubs}

Let $\mcc$ be a category.

\begin{enumerate}

\item Given $X \in \mcc$:

	\begin{enumerate}

	\item Let $\widetilde{\Sub} X$ denote the full subcategory of $\mcc / X$ on the regular monomorphisms of $\mcc$ with codomain $X$.

	\item Let $\Sub X$ denote the (necessarily posetal) category whose objects consist of isomorphism classes of objects of $\widetilde{\Sub} X$, and whose morphisms are the induced equivalence classes of morphisms in $\widetilde{\Sub} X$. Equivalently, $\Sub X$ is any skeleton of $\widetilde{\Sub} X$. By a \emph{subobject} of $X$ we mean an object of $\Sub X$.

	\end{enumerate}

\item Given $f: X \rightarrow Y$, we say that $i: f(X) \rightarrow Y$ is the \emph{r-image of $f$} when $i$ is a regular mono factoring $f = i e$ such that for any other factorization $f = m e^\prime$ with $m$ a regular mono, we have a commutative diagram
\begin{tikzcd}
X
\arrow{r}{e}
\arrow{d}[swap]{e^\prime}
	& \cdot
	\arrow[dotted]{dl}[swap]{\exists !}
	\arrow{d}{i} \\
\cdot
\arrow{r}[swap]{m}
	& Y
\end{tikzcd}.

We say that $\mcc$ \emph{has r-images} when every morphism $f$ has a factorization as above.

\end{enumerate}

\end{definition}

As was the case with monomorphisms, it is an easy fact that regular monos are preserved under pullback. Therefore, any $f: X \rightarrow Y$ induces a pullback functor $\underline{f}^\ast: \Sub Y \rightarrow \Sub X$.

We make the following definition:

\begin{definition}
\label{def:rreg}
\item

\begin{enumerate}

\item A category $\mcc$ is called \emph{r-regular} when it satisfies the following:

\begin{enumerate}

\item $\mcc$ has finite limits;

\item $\mcc$ has r-images;

\item Pulling back across morphisms in $\mcc$ preserves r-images;

\item The composition of two regular monomorphisms in $\mcc$ is again a regular monomorphism in $\mcc$.

\end{enumerate}

\item An r-regular category $\mcc$ is called \emph{r-geometric} when for each $X \in \mcc$, $\Sub X$ is a small-complete lattice and pulling back across any $f: X \rightarrow Y$ preserves this structure.

\end{enumerate}

\end{definition}

We assume that our categories are \emph{r-wellpowered}, which is to say that for each $X \in \mcc$ the objects of $\Sub X$ form a set instead of a proper class.

As is the case with regular and geometric categories, r-regular and r-geometric categories automatically yield the following structure:

\begin{proposition}
\label{prop:heytadjstopb}

Let $\mcc$ be a category.

\begin{enumerate}

\item If $\mcc$ is r-regular, the following hold:\label{rregresults}

	\begin{enumerate}
	
	\item For each $f: X \rightarrow Y$, the pullback functor $\underline{f}^\ast: \Sub Y \rightarrow \Sub X$ has a left adjoint $\underline{\exists}_f: \Sub X \rightarrow \Sub Y$.\label{rregladj}
	
	\item $\underline{\exists}_f (A \wedge \underline{f}^\ast B) = \underline{\exists}_f A \wedge B$ for each $A \in \Sub X$ and $B \in \Sub Y$.\label{frobrec}
	
	\item Given a commutative square \quad
	\begin{tikzcd}
	X
	\arrow{r}{f}
	\arrow{d}{h}
		& Y
		\arrow{d}{g} \\
	W
	\arrow{r}{k}
		& Z
	\end{tikzcd}
	, the induced square
	\begin{equation*}
	\begin{tikzcd}[row sep = 4em, column sep = 4em]
	\Sub Y
	\arrow{r}{\underline{f}^\ast}
	\arrow{d}{\underline{\exists}_g}
		& \Sub X
		\arrow{d}{\underline{\exists}_h} \\
	\Sub Z
	\arrow{r}{\underline{k}^\ast}
		& \Sub W
	\end{tikzcd}
	\end{equation*}
	commutes.\label{bccond}
	\end{enumerate}

\item If in addition $\mcc$ is r-geometric, then the following hold:\label{rgeomresults}

	\begin{enumerate}
	
	\item For each $f: X \rightarrow Y$, the pullback functor $\underline{f}^\ast: \Sub Y \rightarrow \Sub X$ also has a right adjoint $\underline{\forall}_f: \Sub X \rightarrow \Sub Y$.\label{rgeomradj}

	\item For each $X \in \mcc$ and $A \in \Sub X$, the ``meet-with-$A$'' functor $A \wedge - : \Sub X \rightarrow \Sub X$ has a right adjoint $A \Rightarrow - : \Sub X \rightarrow \Sub X$ (``Heyting implication'') which is contravariant in $A$.\label{rgeomheyting}
	
	\end{enumerate}

\end{enumerate}

\end{proposition}

\begin{proof}

The proofs of (\ref{rregladj}) and (\ref{frobrec}) are essentially the same as for regular categories, and can be found in e.g. A1.3 of \cite{johnstone}. (\ref{bccond}) follows easily by pullback stability of r-images.

(\ref{rgeomradj}) is an easy consequence of the definition of r-geometric category and the adjoint functor theorem. The construction that gives (\ref{rgeomheyting}) is found in e.g. 1.7 of \cite{fresce}.

\end{proof}

In particular, the above imply that for an r-geometric category $\mcc$ with $X \in \mcc$, finite meets in $\Sub X$ commute with arbitrary joins.

The structures/results described in Proposition \ref{prop:heytadjstopb} above are important features of categorical semantics of usual first order logic, as described in e.g. \cite{johnstone}, \cite{moemac}. The point emphasized by Proposition \ref{prop:heytadjstopb} is that, just as well as geometric categories, r-geometric categories support the interpretation of first order logic, where we interpret predicates to be \emph{regular} subobjects. Later in this paper we will have occasion to study the adjoints given by (\ref{rregladj}) and (\ref{rgeomradj}) of Proposition \ref{prop:heytadjstopb}.

\subsection{Categorical semantics in $\pmetv$}

The relevance for us of the preceding exposition of r-geometric categories is that $\pmetv$ is an example of such. As a point of terminology, whenever we mention ``modulus of continuity'' it should be understood that we are referring to a modulus of \emph{uniform} continuity.

We first establish the following lemma:

\begin{lemma}
\label{lem:metpbs}

$\pmetv$ has finite limits.

\end{lemma}

\begin{proof}

$\pmetv$ clearly has a terminal object, so it suffices to show that it has pullbacks. Given a diagram
\begin{equation}
\label{eq:pbdiag}
\begin{tikzcd}
 & B
 \arrow{d}{j}\\
X
\arrow{r}{f}
 & Y \\
\end{tikzcd}
\end{equation}
there is a set-theoretic pullback $X \times_{Y} B$, which we take to be the set of points of a space $A$. All the points of $A$ are thus uniquely of the form $(x,b)$ for $x \in X$ and $b \in B$; we take the distance between $(x_1, b_1)$ and $(x_2, b_2)$ to be $\max \left ( d_X (x_1, x_2), d_B ( b_1, b_2) \right )$. There are evident maps $i: A \rightarrow X$ and $h: A \rightarrow B$ given by $i: (x,b) \mapsto x$ and $h: (x,b) \mapsto b$. It is easily verified that this is the required pullback.

\end{proof}

\begin{remark}
\label{rmk:pbofisos}

In case $j$ in the diagram (\ref{eq:pbdiag}) is an isometric embedding, there is a choice of pullback
\begin{equation*}
\label{eq:pbdiag2}
\begin{tikzcd}
A
\arrow{r}{h}
\arrow{d}[swap]{i}
\arrow[phantom, "\lrcorner", very near start]{dr}
 & B
 \arrow{d}{j}\\
X
\arrow{r}{f}
 & Y \\
\end{tikzcd}
\end{equation*}
such that $i$ is also an isometric embedding; let the points of $A$ be the same as before, but take the distance between $(x_1, b_1)$ and $(x_2, b_2)$ to be $d_X(x_1, x_2)$ (one can check that the resulting space is isomorphic to the one constructed above). We still have the maps $i: A \rightarrow X$ and $h: A \rightarrow B$, where we can take the modulus of continuity of $h$ to be the same as that of $f$.

\end{remark}

The below lemma shows that regular monos are closed under composition in $\pmetv$ (which is trivial to verify for $\pmet$, since isomorphisms are isometries there).

\begin{lemma}
\label{lem:regmonoscomp}

Let $l, m$ be composable regular monos in $\pmetv$, i.e. the domain of $m$ and the codomain of $l$ are the same. Then their composition $ml$ is again a regular mono in $\pmetv$.

\end{lemma}

\begin{proof}

Proposition \ref{prop:regmonoembeds} gives us that both $l$ and $m$ are equivalent (in the sense of the paragraph preceding Proposition \ref{prop:regmonoembeds}) to isometric embeddings, giving us the following commutative diagram:
\begin{equation*}
\begin{tikzcd}
	& A
	\arrow{dr}{l^\prime}
		&
			&Y
			\arrow{dr}{m^\prime}
				& \\
\cdot
\arrow{ur}{f}
\arrow{rr}{l}
	&
		& X
		\arrow{ur}{g}
		\arrow{rr}{m}
			&
				& \cdot
\end{tikzcd}
\end{equation*}
where $f,g$ are isomorphisms and $l^\prime, m^\prime$ are isometric embeddings. We now construct a space $B$, with an isomorphism $g^\prime: A \rightarrow B$ and an isometric embedding $k: B \rightarrow Y$ which extends the previous diagram to the following commutative diagram:
\begin{equation*}
\begin{tikzcd}
	&	
		& B
		\arrow{dr}{k}
			&
				& \\
	& A
	\arrow{dr}{l^\prime}
	\arrow{ur}{g^\prime}
		&
			&Y
			\arrow{dr}{m^\prime}
				& \\
\cdot
\arrow{ur}{f}
\arrow{rr}{l}
	&
		& X
		\arrow{ur}{g}
		\arrow{rr}{m}
			&
				& \cdot
\end{tikzcd}
\end{equation*}
which would then exhibit $ml$ as equivalent (via $g^\prime f$) to an isometric embedding $m^\prime k$, therefore a regular mono.

Since $l^\prime: A \rightarrow X$ is an isometric embedding we may consider $A$ to be a subspace of $X$, with $l^\prime$ being the inclusion. Let $B$ be the image of the restriction $g \mid_A$ of $g$ to $A \subseteq X$. Clearly, since $B$ is a subspace of $Y$, there is an isometric embedding (the inclusion) $k: B \rightarrow Y$. Moreover the isomorphism $g: X \rightarrow Y$ restricts to an isomorphism $g^\prime: A \rightarrow B$, since $B$ is the image of the restriction of $g$ to $A$. By construction we have that $gl^\prime = kg^\prime$.

\end{proof}

\begin{proposition}
\label{prop:metheyt}

$\pmetv$ is an r-geometric category.

\end{proposition}

\begin{proof}

Since regular monos are (without loss of generality) isometric embeddings in $\pmetv$, r-images are just the images - in the traditional sense - of functions between spaces. If $f = ie$ is an $r$-image factorization of a map $f: X \rightarrow Y$, then we must show that for any map $g: Z \rightarrow Y$, the map $g^\ast i$ in the following diagram
\begin{equation*}
\begin{tikzcd}
g^\ast X
\arrow{d}{g^\ast e}
\arrow{r}
	& X
	\arrow{d}{e} \\
\cdot
\arrow{d}{g^\ast i}
\arrow{r}
	& \cdot
	\arrow{d}{i} \\
Z
\arrow{r}{g}
	& Y
\end{tikzcd}
\end{equation*}
is an $r$-image for $g^\ast f$.

Now the proof of Lemma \ref{lem:metpbs} (or the existence of an obvious forgetful functor $\pmetv \rightarrow \catsets$) shows that pullbacks in $\pmetv$ agree with the underlying pullbacks in $\catsets$. Therefore $g^\ast i$ at least has the correct underlying set. But pullbacks preserve the property of being a regular mono, so $g^\ast i$ is (equivalent to) an isometric embedding. This is enough to conclude that $g^\ast f = (g^\ast i) (g^\ast e)$ is an $r$-image factorization. Moreover, Lemma \ref{lem:regmonoscomp} shows that regular monos are closed under composition. In this way, $\pmetv$ inherits the regular category structure of $\catsets$ as an r-regular category structure.

Proposition \ref{prop:regmonoembeds} gives us that for each $X \in \pmetv$, $\Sub X$ is naturally isomorphic to the powerset of the underlying set of $X$, so $\pmetv$ automatically inherits the geometric category structure of $\catsets$ as an r-geometric category structure.

\end{proof}

Henceforth when we refer to the image of a map $f: X \rightarrow Y$ in $\pmetv$ we mean the r-image of $f$, which coincides with what is traditionally meant by ``image of a map''.

In the case of $\pmetv$, we have the following explicit descriptions of the adjoints to the pullback functor referred to in Proposition \ref{prop:heytadjstopb}:

Given $f: X \rightarrow Y$,

\begin{enumerate}

\item The left adjoint $\underline{\exists}_f: \Sub X \rightarrow \Sub Y$ to the pullback functor $\underline{f}^\ast$ takes $A \in \Sub X$ to $\underline{\exists}_f A \in \Sub Y$, where $\underline{\exists}_f A$ has the set of points given by the image of $f$. Functoriality of $\underline{\exists}_f$ is immediate.

\item The right adjoint $\underline{\forall}_f: \Sub X \rightarrow \Sub Y$ takes $A \in \Sub X$ to $\underline{\forall}_f A \in \Sub Y$, where $\underline{\forall}_f A$ is $\{ y \in Y \mid \text{$A$ contains all $x$ for which $f(x) = y$} \}$. Again, functoriality of $\underline{\forall}_f$ is immediate.

\end{enumerate}

It is straightforward to verify that the above indeed give left and right adjoints to $\underline{f}^\ast$.

That these constructions are the same as in the category of sets is not coincidence; since $\pmetv$ directly inherits its r-geometric structure from the geometric structure of $\catsets$, we should expect the various logical operations that may involve (strong) subobjects of spaces $X \in \pmetv$ should agree with those same operations performed on the subobjects of the underlying sets of those spaces.

\section{Indexed subobjects as predicates of continuous logic}
\label{sec:indexedsubs}

We continue investigating the categorical properties of $\pmetv$, with the intermediate goal of showing that $\pmetv$ possesses a ``continuous subobject classifier'', where by ``continuous subobject'' we mean a real-valued predicate satisfying appropriate continuity conditions. The motivation for our specific approach in this section is as follows.

Lawvere's insights in \cite{lawvere} made concrete the idea that the notion of distance is an immediate generalization of the equality predicate to the setting of $[0,1]$-valued logic (thinking of $0$ as ``true'' and $1$ as ``false''), where e.g. the triangle inequality \emph{is} the transitivity of equality. Specifically, monoidal categories $(\mcv, \otimes, \text{unit})$ may be thought of as giving either truth values for the equality predicate in appropriate corresponding logics, or values for appropriate notions of distance. From this perspective, $\mcv$-categories, i.e. categories whose hom-objects are objects of $\mcv$, are ``metric spaces'' with distances valued in $\mcv$, where the objects of the $\mcv$-category are points and the hom-objects between them are distances; equivalently, the ``degree of equality'' between two objects in a $\mcv$-category is given by their hom-object. The monoidal product of $\mcv$ - which governs composition of hom-objects in $\mcv$-categories - thus gives the notion of composition of ``distance'' or ``equality''. This perspective concretely manifests as the following.

Depending on the choice of $\mcv$, these $\mcv$-categories are actual metric spaces (if $\mcv = \mathbb{R}_{\geq 0}$ or $\mathbb{I}$; see below), where:
\begin{enumerate}[label=$\circ$, ref=$\circ$]

\item distance/equality is real-valued;

\item with composition given by (truncated, if $\mcv = \mathbb{I}$) addition;

\item the composition axioms for $\mcv$-categories yield the triangle inequality;

\end{enumerate}

or simply just sets (if $\mcv = \mathbbm{2}$; see below), where
\begin{enumerate}[label=$\circ$, ref=$\circ$]

\item distance/equality is $\{0,1\}$-valued;

\item with composition given by conjunction;

\item the composition axioms for $\mcv$-categories yield transitivity of equality.

\end{enumerate}

The prototypical example of a monoidal category relevant to logic is $(\mathbbm{2}, \max, 0)$, where $\mathbbm{2}$ is the poset category $0 \leftarrow 1$, and we think of $0$ as ``true'' and $1$ as ``false'' because we think of these as values for the ``distance'' i.e. equality predicate. (This notational convention, while sensible from our ``metric'' perspective, suffers the unfortunate side effect that $0$ is the terminal object and $1$ is the initial object.) As mentioned above, $\mathbbm{2}$-categories are essentially just sets, and the corresponding logic is the usual one where equality is $\{0,1\}$-valued.

The relevant monoidal category for metric spaces and continuous logic is either $(\mathbb{R}_{\geq 0}, +, 0)$ or $(\mathbb{I}, \dotplus, 0)$, where $\mathbb{R}_{\geq 0}$ is the poset category with:

\begin{enumerate}[label=$\circ$,ref=$\circ$]

\item set of objects the set $[0,\infty]$;

\item morphisms $a \leftarrow b$ iff $a \leq b$ as real numbers;

\end{enumerate}

and $\mathbb{I}$ is the poset category with:

\begin{enumerate}[label=$\circ$,ref=$\circ$]

\item set of objects the set $[0,1]$;

\item morphisms $a \leftarrow b$ iff $a \leq b$ as real numbers;

\item $a \dotplus b = \min (a+b, 1)$.

\end{enumerate}

To minimize notational clutter we will simply use $+$ to mean $\dotplus$ when it is clear that we are working with $\mathbb{I}$. The obvious inclusion of monoidal categories $\mathbbm{2} \hookrightarrow \mathbb{I}$ induces a functor turning each $\mathbbm{2}$-category into an $\mathbbm{I}$-category, which corresponds to regarding every set as a metric space with the discrete metric.

The lesson embedded in \cite{lawvere} is that taking an appropriately formulated aspect of categorical semantics of the usual $\{0,1\}$-valued logic - which describes categories in terms of how similar they are to the category of sets, the quintessential model of $\{0,1\}$-valued logic - and carefully replacing occurrences of $\mathbbm{2} = \{0,1\}$ with $\mathbbm{I} = [0,1]$ should yield an analogous aspect of categorical semantics of $[0,1]$-valued logic, which should describe categories in terms of how similar they are to the category of metric spaces. This is the heuristic that underpins the approach we take in this section.

In what follows, $\mcv$ will always denote a (complete) posetal monoidal category with appropriate structure, two important examples of which are $\mathbbm{2}$ and $\mathbb{I}$; such a category is called a \emph{quantale}, whose definition we recall below. In accordance with these main examples, for any $r, s \in \mcv$ when we say $r \leq s$ we mean that $r \leftarrow s$ in $\mcv$, and moreover by $\inf\limits_i r_i$ (resp. $\sup\limits_i r_i$) we mean the \emph{join} (resp. \emph{meet}) of the $r_i$ taken in $\mcv$ (but $A \leq B$ in $\Sub X$ still means
\begin{tikzcd}
A
\arrow[tail]{r}
	& B
\end{tikzcd}). The point is that the framework we will describe applies in this more general setting; in particular everything that follows in the rest of the paper works equally well when we replace $\mathbb{I}$ (i.e. $[0,1]$) with $\mathbb{R}_{\geq 0}$ (i.e. $[0, \infty]$).

\begin{definition}
\label{def:quantale}

By a \emph{quantale} we mean a (small) posetal monoidal category $(\mcv, \otimes, \text{unit})$ which possesses all meets and joins, such that $\otimes$ preserves joins. We further require that the terminal object of $\mcv$ is the unit for $\otimes$ (these are sometimes called \emph{affine quantales}).

\end{definition}

\subsection{Continuity from indexed subobjects}
\label{subsec:indexed}

In the case of classical logic and the category $\catsets$, a predicate on $X$ may equivalently be considered as a subobject $R \in \Sub X$, or a function $\chi_R: X \rightarrow \Omega = | \mathbbm{2} |$, where $| \mathbbm{2} |$ is the underlying set $\{0, 1 \}$ of the monoidal category $\mathbbm{2}$ of classical truth values, and $\chi_R(x) = 0$ iff $x \in R \subseteq X$. (In the reverse direction we have that any function $\chi: X \rightarrow | \mathbbm{2} |$ gives a subobject $R_{\chi}$ given by $x \in R_{\chi}$ iff $\chi (x) = 0$.)

Now each subobject $A \in \Sub X$ may alternatively be thought of as an object $R \in [\Twop, \Sub X]$ of a special form, by setting $R(0) = A$ and $R(1)$ to be the terminal object in $\Sub X$, i.e. the whole set $X$. By this trivial change in perspective, in the classical case (with $\mathbbm{2}$-valued logic) we have a correspondence between maps $X \rightarrow | \mathbbm{2} |$ and (objects of) the full subcategory $\TSub X \subseteq [\Twop, \Sub X]$ on the objects $R \in [\Twop, \Sub X]$ for which $R(1)$ is terminal.

Let $| \mathbb{I} |$ be the real interval $[0, 1]$, i.e. the ``underlying space'' of the monoidal category $\mathbb{I}$, considered as an object of $\pmetv$ (with the obvious metric). We will see that by suitably extending the above perspective, we get a correspondence between maps $X \rightarrow | \mathbb{I} |$ in $\pmetv$ and (the objects of) an appropriate full subcategory of $[\Iop, \Sub X]$.

Each $X \in \pmetv$ comes with a distinguished object $D_X \in [\Iop, \Sub (X \times X)]$ which encodes the metric; we define $D_X(r)$ as giving all the pairs $(x,y) \in X \times X$ for which the distance $d_X(x,y)$ between $x$ and $y$ is $\leq r$. Here we are implicitly making a choice of product $X \times Y$ given $X, Y \in \pmetv$; we are taking $X \times Y$ to be the space with the obvious set of points and metric given by $d_{X \times Y} ( (x_1, y_1), (x_2, y_2) ) = \max ( d_X (x_1, x_2), d_Y (y_1, y_2) )$.

Categorically, the important properties of $D_X$ (whose straightforward proofs are left to the reader) are as follows:

\begin{proposition}
\label{prop:metricprops}

For each $X, Y \in \pmetv$ there is a choice of product $X \times Y$, well-behaved in the evident sense with respect to symmetry and associativity, such that there is a distinguished object $D_X \in [\Iop, \Sub (X \times X)]$ satisfying:

\begin{enumerate}

\item $D_X(0)$ contains the diagonal;

\item The functor $[\Iop, \Sub (X \times X)] \xrightarrow{\cong} [\Iop, \Sub (X \times X)]$ induced by the symmetry isomorphism $X \times X \xrightarrow{\cong} X \times X$ interchanging the factors takes $D_X$ to itself;

\item Let $\pi_{i,j}: (X \times X \times X) \rightarrow (X \times X)$ denote the projection onto the $i^{\text{th}}$ and $j^{\text{th}}$ factors, respectively. Then $\underline{\pi_{i,j}}^{\ast} D_X(r) \wedge \underline{\pi_{j,k}}^{\ast} D_X(s) \leq \underline{\pi_{i,k}}^{\ast} D_X (r + s)$ for every $r, s \in \mathbb{I}$.

\item If $r = \inf\limits_i r_i$ for some $r, r_i \in \mathbb{I}$, then $D_X(r) = \bigwedge\limits_i D_X (r_i)$;

\item Let $\pi_{X \times X}: (X \times Y \times X \times Y) \rightarrow (X \times X)$ and $\pi_{Y \times Y}: (X \times Y \times X \times Y) \rightarrow (Y \times Y)$ denote the projections preserving the ordering of the factors. Then $D_{X \times Y}(r) = \underline{\pi_{X \times X}}^{\ast} D_X(r) \wedge \underline{\pi_{Y \times Y}}^{\ast} D_Y(r)$ for all $r \in \mathbb{I}$.\label{prodmetric}

\end{enumerate}

\end{proposition}

Henceforth when we say ``the product $X \times Y$'' for some $X, Y$, we are referring to this choice of product for $X \times Y$. (For any of the variations of $\pmetv$, this choice is just the maximum metric.)

\begin{remark}
\label{rmk:evil}

These are the beginnings of our framework, but we make an observation that may make us uncomfortable, namely that aspects of our framework fail to be isomorphism invariant (unless we work only in $\pmet$). That is, in $\pmetu$ or $\pmetl$ we may have an isomorphism $f: X \xrightarrow{\cong} X^\prime$, but will not in general have that $\underline{f \times f}^{\ast} D_{X^{\prime}} (r) = D_X (r)$ for any given $r \in \mathbb{I}$.

Indeed, the price we pay for translating the metric $d_X$ on each $X \in \pmetv$ into categorical data $D_X \in [\Iop, \Sub (X \times X)]$ is that we must keep track of the actual values of the metric on $X$, whereas isomorphisms do not necessarily respect the metric (again, unless we are in $\pmet$ where isomorphisms are necessarily isometries). This is the reason for the somewhat peculiar phrasing at the beginning of Proposition \ref{prop:metricprops} above. The upshot is that as long as we are sufficiently careful this does not end up being a problem.

\end{remark}

Now maps $X \rightarrow Y$ in $\pmetu$ are just set functions which have (and obey) a \emph{modulus of continuity}, which is an increasing function $\epsilon: [0, 1] \rightarrow [0, 1]$ that preserves infima and satisfies $\epsilon(0) = 0$. That is, a map $f: X \rightarrow Y$ is just a set function $f$ between the underlying sets such that $d_Y ( f(x), f(x^\prime) ) \leq \epsilon \left ( d_X (x, x^\prime) \right )$ for all pairs of points $x, x^\prime \in X$, where $\epsilon$ is as above. It is important that being a modulus of continuity is preserved under composition: if $\epsilon_1$ and $\epsilon_2$ are moduli of continuity, then $\epsilon_2 \circ \epsilon_1$ is again a modulus of continuity.

More categorically, this amounts to the following. Each modulus of continuity $\epsilon$ as above is equivalently an endofunctor $\epsilon: \mathbb{I} \rightarrow \mathbb{I}$ which fixes $0$ and is cocontinuous (preserves joins). Considering the endofunctor category $\Endo (\mathbb{I})$ as a monoid with the multiplication given by composition, the set $E = E_u$ of all $\epsilon$ satisfying the above gives a submonoid $E \subseteq \Endo (\mathbb{I})$. Changing our choice of $E$ corresponds to changing our choice of the variant of $\pmetv$ we are working with, since our choice of $E$ dictates which moduli of continuity are allowed; see Example \ref{ex:moduloids} and Lemma \ref{lem:continuity} below.

We note that this perspective generalizes to any quantale $\mcv$ in place of $\mathbb{I}$. We thus formalize the discussion above as follows.

\begin{definition}
\label{def:genmoduli}

Let $\mcv$ be a quantale. A functor $\epsilon: \mcv \rightarrow \mcv$ which preserves joins and the monoidal unit is called a \emph{$\mcv$-modulus}.

A (full) subcategory $E$ of the functor category $\Endo({\mcv})$ on a set of $\mcv$-moduli containing the identity $1_{\mcv}$ and which is closed under functor composition, i.e. a submonoid of $\Endo(\mcv)$, is called a \emph{$\mcv$-moduloid}.

\end{definition}

Depending on context, we consider $\mcv$-moduli $\epsilon: \mcv \rightarrow \mcv$ equivalently as functors $\epsilon: \Vop \rightarrow \Vop$; this will always be clear from context. Moreover, we refer to joins (resp. meets) in $\mcv$ as minima/infima (resp. maxima/suprema) so $\mcv$-moduli are functors which preserve infima and the monoidal unit. We extend our notational convention by saying that $\epsilon_1 \leq \epsilon_2$ when $\epsilon_1 \leftarrow \epsilon_2$ in the functor category $[\mcv, \mcv]$. Thus $\max ( \epsilon_1, \epsilon_2 )$ refers to the \emph{meet} of $\epsilon_1$ and $\epsilon_2$ as objects of $[\mcv, \mcv]$.

Important examples of quantales are the Lawvere real numbers $\mathbb{R}_{\geq 0} = (\mathbb{R}_{\geq 0}, +, 0)$ and their truncation $\mathbb{I} = (\mathbb{I}, +, 0)$, both of which we have seen at the beginning of the current section. Although we continue to focus on the latter, all of our analysis applies equally well using the former instead.

\begin{example}
\label{ex:moduloids}

We may consider the following $\mathbb{I}$-moduloids:
\begin{enumerate}

\item Let $E_u$ contain all $\mathbb{I}$-moduli. This collection certainly includes $1_{\mathbb{I}}$ and is closed under composition, so it is the maximal $\mathbb{I}$-moduloid.

\item Let $E_L$ contain all the $\mathbb{I}$-moduli given by ``multiplication by a (finite) constant'', i.e. all $\epsilon_K: \mathbb{I} \rightarrow \mathbb{I}$ of the form
\begin{equation*}
r \mapsto \min (Kr, 1)
\end{equation*}
for a finite constant $K \geq 1$. Then $\epsilon_{K_2} \circ \epsilon_{K_1} = \epsilon_{K_2 K_1}$, so that $E_L$ is a $\mathbb{I}$-moduloid.

\item Let $E_1$ contain only $1_{\mathbb{I}}$, so that it is the trivial $\mathbb{I}$-moduloid.

\end{enumerate}

\end{example}


In this light, we have the following alternate characterization of continuity:

\begin{lemma}
\label{lem:continuity}
\item

\begin{enumerate}

\item

	\begin{enumerate}

	\item For any map $f: X \rightarrow Y$ in $\pmetu$, there is some $\epsilon: \mathbb{I} \rightarrow \mathbb{I}$ in $E_u$ such that $\epsilon$ is a modulus of continuity for $f$.

	\item A map $f: X \rightarrow Y$ in $\pmetu$ is continuous with respect to a modulus $\epsilon: \mathbb{I} \rightarrow \mathbb{I}$ in $E_u$ iff for every $r\in \mathbb{I}$, $D_X (r) \leq \underline{f \times f}^{\ast} D_Y ( \epsilon(r) )$.

	\end{enumerate}

\item

	\begin{enumerate}

	\item For any map $f: X \rightarrow Y$ in $\pmetl$, there is some $\epsilon: \mathbb{I} \rightarrow \mathbb{I}$ in $E_L$ such that $\epsilon$ is a modulus of continuity for $f$.

	\item A map $f: X \rightarrow Y$ in $\pmetl$ is continuous with respect to a modulus $\epsilon: \mathbb{I} \rightarrow \mathbb{I}$ in $E_L$ iff for every $r\in \mathbb{I}$, $D_X (r) \leq \underline{f \times f}^{\ast} D_Y ( \epsilon(r) )$.

	\end{enumerate}

\item

	\begin{enumerate}

	\item For any map $f: X \rightarrow Y$ in $\pmet$, the only object $1_{\mathbbm{I}}: \mathbb{I} \rightarrow \mathbb{I}$ of $E_1$ is a modulus of continuity for $f$.

	\item A map $f: X \rightarrow Y$ in $\pmet$ is continuous with respect to the only modulus $1_{\mathbb{I}}: \mathbb{I} \rightarrow \mathbb{I}$ in $E_1$ iff for every $r\in \mathbb{I}$, $D_X (r) \leq \underline{f \times f}^{\ast} D_Y ( r )$.

	\end{enumerate}

\end{enumerate}

\end{lemma}

\begin{proof}

These follow easily from the definitions of $\pmetu$, $\pmetl$, $\pmet$ and the $\mathbbm{I}$-moduloids $E_u$, $E_L$, and $E_1$, as well as the statement that $f:X \rightarrow Y$ is continuous with respect to a modulus $\epsilon: \mathbbm{I} \rightarrow \mathbbm{I}$ iff for all $r \in \mathbbm{I}$, $d_X(x,y) \leq r \Longrightarrow d_Y(f(x),f(y)) \leq \epsilon(r)$.

\end{proof}

As we did with taking $\pmetv = \pmetu$, we make the choice of $\mathbb{I}$-moduloid $E = E_u$ just to be precise, while keeping in mind that most of what follows works equally well if we choose ($\pmetv = \pmetl$ and) $E = E_L$ or ($\pmetv = \pmet$ and) $E = E_1$ instead. Thus when we speak of $E$ without mentioning $E_u$ specifically, or when we speak of ``an $\mathbbm{I}$-modulus $\epsilon: \mathbbm{I} \rightarrow \mathbbm{I}$'' without specifying $\epsilon \in E_u$ specifically, it should be understood that everything applies equally well for $E_L$ and $E_1$. In the cases where the choice of $E$ makes a difference, we will say so explicitly.

\begin{proposition}
\label{prop:formalpfs}

Only the formal properties - in fact, just (\ref{prodmetric}) - of the metric listed in Proposition \ref{prop:metricprops} are already sufficient to imply the following properties of $\pmetv$:

\begin{enumerate}

\item For $X \xrightarrow{f} Y \xrightarrow{g} Z$ such that $f$ and $g$ have moduli $\epsilon_f$ and $\epsilon_g$ respectively, we have that $gf: X \rightarrow Z$ has modulus $\epsilon_g \circ \epsilon_f$.\label{compmodulus}

\item For any $X, Y$, $\pi_X: X \times Y \rightarrow X$ satisfies
\begin{equation*}
D_{X \times Y} (r) \leq \underline{\pi_X \times \pi_X}^{\ast} D_X (r)
\end{equation*}
for all $r \in \mathbb{R}$.\label{projmodulus}

\item Each pair of maps $f: X \rightarrow Y$, $g: X \rightarrow Z$ (with moduli $\epsilon_f$ and $\epsilon_g$ respectively) canonically corresponds to the obvious map $(f,g): X \rightarrow (Y \times Z)$, with modulus $\epsilon_{(f,g)} = \max (\epsilon_f, \epsilon_g)$.

In the other direction, if $\epsilon_{(f,g)}$ is a modulus for $(f,g)$ then it is also a modulus for both $f$ and $g$.\label{prodmodulus}

\end{enumerate}

\end{proposition}

\begin{proof}

(\ref{compmodulus}): $D_X(r) \leq \underline{f \times f}^{\ast} D_Y(\epsilon_f (r)) \leq \underline{f \times f}^{\ast} \underline{g \times g}^{\ast} D_Z(\epsilon_g (\epsilon_f (r))) = \underline {gf \times gf} ^{\ast} D_Z ( \epsilon_g \circ \epsilon_f (r) )$.

(\ref{projmodulus}): $\pi_X \times \pi_X: X \times Y \times X \times Y \rightarrow X \times X$ is just $\pi_{X \times X}$. We have
\begin{equation*}
D_{X \times Y} (r) = \underline{\pi_{X \times X}}^{\ast} D_X(r) \wedge \underline{\pi_{Y \times Y}}^{\ast} D_Y(r) \leq \underline{\pi_{X \times X}}^\ast D_X (r) = \underline{\pi_X \times \pi_X}^{\ast} D_X (r).
\end{equation*}

(\ref{prodmodulus}): Ignoring moduli, there is obviously a correspondence between pairs of maps $f: X \rightarrow Y$, $g: X \rightarrow Z$ and $(f,g): X \rightarrow (Y \times Z)$. Now assume that $D_X(r) \leq \underline{f \times f}^{\ast} D_Y ( \epsilon_f (r) )$ and $D_X(r) \leq \underline{g \times g}^{\ast} D_Z ( \epsilon_g (r) )$ for all $r \in \mathbb{R}$.

We want to show
\begin{equation*}
D_X (r) \leq \underline {(f, g) \times (f, g)}^{\ast} D_{Y \times Z} (\epsilon_{(f,g)} (r) )
\end{equation*}
where $\epsilon_{(f,g)} = \max (\epsilon_f, \epsilon_g)$. For convenience denote $\epsilon_{(f,g)}(r) = s$.
\begin{equation*}
\underline {(f, g) \times (f, g)}^{\ast} D_{Y \times Z} ( s ) = \underline {(f, g) \times (f, g)}^{\ast} \left ( \underline{ \pi_{Y \times Y}}^{\ast} D_Y(s) \wedge \underline{ \pi_{Z \times Z} }^{\ast} D_Z (s) \right )
\end{equation*}
We have $f = \pi_Y (f,g)$ and $g = \pi_Z (f,g)$ and $\pi_{Y \times Y} = \pi_Y \times \pi_Y$ (resp. $\pi_{Z \times Z} = \pi_Z \times \pi_Z$), so the above turns into
\begin{equation*}
\underline {(f, g) \times (f, g)}^{\ast} \left ( \underline{ \pi_{Y \times Y}}^{\ast} D_Y(s) \wedge \underline{ \pi_{Z \times Z} }^{\ast} D_Z (s) \right ) = (\underline{f \times f}^{\ast} D_Y (s) ) \wedge (\underline{g \times g}^{\ast} D_Z (s) )
\end{equation*}
but $D_X(r) \leq (\underline{f \times f}^{\ast} D_Y (s) ) \wedge (\underline{g \times g}^{\ast} D_Z (s) )$.

The same proof shows that in the other direction (i.e. given $(f,g)$ with modulus $\epsilon_{(f,g)}$), we can take $\epsilon_f = \epsilon_g = \epsilon_{(f,g)}$.

\end{proof}

In the case that $Y = | \mathbb{I} | = [0,1]$, with $D_{|\mathbb{I}|}$ corresponding to the obvious metric on $|\mathbb{I}|$, we can specify the continuity of a map $f: X \rightarrow Y$ in terms of (the metric on $X$ and) indexed subobjects of $X$. The rest of this subsection is motivated by the following key observation:

\begin{lemma}
\label{lem:iff}

Let $\mct_{\mathbb{I}} \in [\Iop, \Sub |\mathbb{I}|]$ be defined by $\mct_{\mathbb{I}}(r) = [0,r]$, and let $X \in \pmetv$. Let $\pi_1, \pi_2: (X \times X) \rightarrow X$ denote the projection onto the first and second factors, respectively.

A map $f: X \rightarrow |\mathbb{I}|$ is continuous with respect to a modulus $\epsilon: \mathbb{I} \rightarrow \mathbb{I}$ iff for all $r, s \in \mathbb{R}$, we have that
\begin{equation*}
\left ( \underline{\pi_1}^{\ast} \underline{f}^{\ast} \mct_{\mathbb{I}} (r) \right ) \wedge D_X (s) \leq \underline{\pi_2}^{\ast} \underline{f}^\ast \mct_{\mathbb{I}} (r + \epsilon(s)).
\end{equation*}

\end{lemma}

\begin{proof}

Unpacking the left hand side, $\left ( \underline{\pi_1}^{\ast} \underline{f}^{\ast} \mct_{\mathbb{I}} (r) \right ) \wedge D_X (s)$ consists of all pairs $(x, y) \in X \times X$ such that $f(x) \leq r$ and $d_X(x, y) \leq s$. The right hand side is the set of all pairs $(x,y) \in X \times X$ such that $f(y) \leq r + \epsilon(s)$.

Continuity of $f: X \rightarrow |\mathbb{I}|$ with respect to a modulus $\epsilon: \mathbb{I} \rightarrow \mathbb{I}$ obviously implies the inequality.

Conversely, assume that the inequality holds, i.e. that for all $x,y \in X$ and $r, s \in \mathbb{I}$, whenever $f(x) \leq r$ and $d_X(x,y) \leq s$, we also have $f(y) \leq r + \epsilon(s)$.

Let $s \in \mathbb{I}$ be arbitrary, and let $x, y \in X$ be any pair of points satisfying $d_X(x,y) \leq s$. Without loss of generality, let $r = f(x) \leq f(y)$. By assumption we have that $r \leq f(y) \leq r+\epsilon(s)$, therefore the distance between $f(x)$ and $f(y)$ is bounded above by $\epsilon(s)$.

\end{proof}

\subsection{Continuous subobjects}

The preceding lemma suggests that the following is the appropriate notion of ``continuous subobject''.

\begin{definition}
\label{def:econtsubobj}

Let $X \in \pmetv$, and let $\epsilon: \mathbb{I} \rightarrow \mathbb{I}$ be an $\mathbbm{I}$-modulus.

We call $R \in [\Iop, \Sub X]$ an \emph{$\epsilon$-predicate on $X$} when:

\begin{enumerate}

\item Given $r, r_i \in \mathbb{I}$ such that $r = \inf\limits_i r_i$, $R(r) = \bigwedge\limits_i R(r_i)$, and;

\item For each $r,s \in \mathbb{I}$,
\begin{equation*}
\left ( \underline{\pi_1}^{\ast} R (r) \right ) \wedge D_X (s) \leq \underline{\pi_2}^{\ast} R (r + \epsilon(s))
\end{equation*}

\end{enumerate}

We denote by $\Sube X$ the full subcategory of $[\Iop, \Sub X]$ on the $\epsilon$-predicates.

Let us call $R \in [\Iop, \Sub X]$ a \emph{predicate on $X$} when there exists some $\epsilon \in E$ for which $R$ is an $\epsilon$-predicate.

\end{definition}

Clearly, for $\epsilon_1 \leq \epsilon_2$ any $\epsilon_1$-predicate is also an $\epsilon_2$-predicate, so we have a full inclusion $\Subeo X \hookrightarrow \Subet X$.

The point of Definition \ref{def:econtsubobj} above is to consider $\epsilon$-predicates on $X$ as an intrinsic notion of a uniformly continuous map $X \rightarrow |\mathbb{I}|$ with modulus $\epsilon$. Now given a map $f: X \rightarrow Y$ we have that the pullback functor $\underline{f}^\ast: \Sub Y \rightarrow \Sub X$ acts pointwise to give a functor $\overline{f}^\ast: [\Iop, \Sub Y] \rightarrow [\Iop, \Sub X]$, so that by definition $(\overline{f}^\ast R) (r) = \underline{f}^\ast \left ( R (r) \right )$; we are to think of this as precomposition by $f$.

Thus if we have $R \in \Sube Y$ and a map $f: X \rightarrow Y$ with modulus $\epsilon_f$, we should hope that $\overline{f}^\ast R$ (where $\overline{f}^\ast$ acts on $R$ considered as an object of $[\Iop, \Sub X]$) is an $(\epsilon \circ \epsilon_f)$-predicate, which we now show:

\begin{proposition}
\label{prop:pbofcontsubobj}

Given $f: X \rightarrow Y$ with modulus of continuity $\epsilon_f$, and given $R \in [\Iop, \Sub Y]$ which is an $\epsilon$-predicate, we have that $\overline{f}^\ast R \in [\Iop, \Sub X]$ is an $(\epsilon \circ \epsilon_f)$-predicate.

In particular, $\overline{f}^\ast: [\Iop, \Sub Y] \rightarrow [\Iop, \Sub X]$ descends to a functor $f^\ast: \Sube Y \rightarrow \Subeef X$ for which $i_{\epsilon \circ \epsilon_f} f^{\ast} = \overline{f}^{\ast} i_{\epsilon}$.

\end{proposition}

\begin{proof}

For $r, r_i \in \mathbb{R}$ with $r = \inf\limits_i r_i$, $\overline{f}^{\ast}R (r) = \bigwedge\limits_i \overline{f}^{\ast} R(r_i)$ since $\overline{f}^{\ast}$ preserves meets.

For any $r,s \in \mathbb{R}$,
\begin{align*}
\left ( \underline{\pi_1}^{\ast} \underline{f}^{\ast} R (r) \right ) \wedge D_X (s)
	&\leq \left ( \underline{\pi_1}^{\ast} \underline{f}^{\ast} R (r) \right ) \wedge \underline{f \times f}^{\ast} D_Y (\epsilon_f(s) ) \\
	& = \left ( \underline{f \times f}^{\ast} \underline{\pi_1}^{\ast} R (r) \right ) \wedge \underline{f \times f}^{\ast} D_Y (\epsilon_f(s) ) \\
	& \leq \underline{f \times f}^{\ast} \underline{\pi_2}^{\ast} R(r + \epsilon ( \epsilon_f (s) ) ) \\
	& = \underline{\pi_2}^{\ast} \underline{f}^{\ast} R (r + \epsilon \circ \epsilon_f (s) ).
\end{align*}

\end{proof}

\begin{example}
\label{ex:metricandtrvals}



$\mct_{\mathbb{I}} \in [\Iop, \Sub |\mathbb{I}|]$ as defined in Lemma \ref{lem:iff}, i.e. $\mct_{\mathbb{I}}(r) = [0,r]$, is a $1_{\mathbb{I}}$-predicate.


\end{example}

We give a more substantial example of predicates, namely, $D_X$ itself is a predicate on $X$ for each $X \in \pmetu$ or $X \in \pmetl$. For this we must make use of a condition on our set $E$ of moduli (which holds in $E_u$ and $E_L$, but not $E_1$). The monoidal product $\otimes$ of $\mcv$ is inherited by the functor category $[\mcv, \mcv]$ just by pointwise application, and so we can ask that our set $E$ of moduli, in addition to being closed under composition, also be closed under $\otimes$, i.e. that if $\epsilon_1, \epsilon_2 \in E$ then $\epsilon_1 \otimes \epsilon_2 \in E$; for now we stick to $\mcv = \mathbb{I}$ for concreteness, so in this case $\otimes$ is just $+$.

The reason that this condition is desirable is that for any $X$, we have that $D_X \in [\Iop, \Sub (X \times X)]$ is in general a $(1_{\mathbb{I}} + 1_{\mathbb{I}})$-predicate on $X \times X$:

\begin{proposition}
\label{prop:metricaspred}

\item

\begin{enumerate}

\item $D_X$ is a $(1_{\mathbb{I}} + 1_{\mathbb{I}})$-predicate on $X \times X$.\label{metricas2pred}

\item Let $f,g: X \rightarrow Y$ be maps with moduli $\epsilon_f$, $\epsilon_g$ respectively. Then $(f,g)^{\ast} D_Y$ is an $(\epsilon_f + \epsilon_g)$-predicate on $X$.\label{fgmetricaspred}

\end{enumerate}

\end{proposition}

\begin{proof}

(\ref{metricas2pred}): We automatically have that for any $r, r_i \in \mathbb{I}$ with $r = \inf\limits_i r_i$ $D_X(r) = \bigwedge\limits_i D_X (r_i)$ since $D_X$ is a metric.

We want to show that for all $r, s \in \mathbb{I}$, we have
\begin{equation*}
\underline{\pi_1}^{\ast} D_X (r) \wedge D_{X \times X} (s) \leq \underline{\pi_2}^{\ast} D_X (r + (1_{\mathbb{I}} + 1_{\mathbb{I}}) (s) ) = \underline{\pi_2}^{\ast} D_X (r + s + s).
\end{equation*}
In order to avoid notational confusion, let us denote by $\pi_{i,j}: (X \times X \times X \times X) \rightarrow (X \times X)$ the projection onto the $i^{\text{th}}$ and $j^{\text{th}}$ factors respectively. Then the above becomes
\begin{equation*}
\underline{\pi_{1,2}}^{\ast} D_X (r) \wedge D_{X \times X} (s) \leq \underline{\pi_{3,4}}^{\ast} D_X (r + s + s)
\end{equation*}
Now $D_{X \times X} (s) = \underline{\pi_{1,3}}^{\ast} D_X(s) \wedge \underline{\pi_{2,4}}^{\ast} D_X(s)$ by Proposition \ref{prop:metricprops}, so we need to show
\begin{equation*}
\underline{\pi_{1,2}}^{\ast} D_X (r) \wedge \underline{\pi_{1,3}}^{\ast} D_X(s) \wedge \underline{\pi_{2,4}}^{\ast} D_X(s) \leq \underline{\pi_{3,4}}^{\ast} D_X (r + s + s)
\end{equation*}
But by Proposition \ref{prop:metricprops} again we have
\begin{equation*}
\underline{\pi_{1,2}}^{\ast} D_X (r) \wedge \underline{\pi_{1,3}}^{\ast} D_X(s) \leq \underline{\pi_{2,3}}^{\ast} D_X(r + s)
\end{equation*}
and
\begin{equation*}
\underline{\pi_{2,3}}^{\ast} D_X (r + s) \wedge \underline{\pi_{2,4}}^{\ast} D_X(s) \leq \underline{\pi_{3,4}}^{\ast} D_X (r + s + s)
\end{equation*}
so we are done.

(\ref{fgmetricaspred}): By (\ref{metricas2pred}) we at least have that $(f,g)^{\ast} D_Y$ is a $(\max(\epsilon_f, \epsilon_g) + \max(\epsilon_f, \epsilon_g) )$-predicate on $X$, so it only remains to show that it admits $\epsilon_f + \epsilon_g$ as a modulus of continuity; we thus want
\begin{equation*}
\underline{\pi_1}^{\ast}\underline{(f,g)}^{\ast} D_Y (r) \wedge D_X(s) \leq \underline{\pi_2}^{\ast} \underline{(f,g)}^{\ast} D_Y(r + \epsilon_f (s) + \epsilon_g (s))
\end{equation*}
for all $r, s \in \mathbb{I}$ (recall that $\left ( (f,g)^{\ast}D_Y \right ) (r)$ is just $\underline{(f,g)}^{\ast} \left ( D_Y (r) \right )$).

Now we have the commutative diagrams
\begin{equation*}
\begin{tikzcd}[row sep = 4em, column sep = 4em]
X \times X
\arrow{r}{(f,g) \times (f,g)}
\arrow[shift right]{d}[swap]{\pi_1}
\arrow[shift left]{d}{\pi_2}
	& Y \times Y \times Y \times Y
	\arrow[shift right]{d}[swap]{\pi_{1,2}}
	\arrow[shift left]{d}{\pi_{3,4}} \\
X
\arrow{r}{(f,g)}
	& Y \times Y
\end{tikzcd}
\end{equation*}
so the above is equivalent to showing
\begin{equation*}
\underline{(f,g) \times (f,g)}^{\ast} \underline{\pi_{1,2}}^{\ast} D_Y (r) \wedge D_X(s) \leq \underline{(f,g) \times (f,g)}^{\ast} \underline{\pi_{3,4}}^{\ast} D_Y(r + \epsilon_f (s) + \epsilon_g (s))
\end{equation*}

Now $D_X(s) \leq \underline{f \times f}^{\ast} D_Y (\epsilon_f (s) )$ and $D_X(s) \leq \underline{g \times g}^{\ast} D_Y (\epsilon_g (s) )$, while $f \times f = \pi_{1,3} \circ \left ( (f,g) \times (f,g) \right )$ and $g \times g = \pi_{2,4} \circ \left ( (f,g) \times (f,g) \right )$.

We thus have
\begin{align*}
\underline{(f,g) \times (f,g)}^{\ast} \underline{\pi_{1,2}}^{\ast} D_Y (r) &\wedge D_X(s) \\
\leq \underline{(f,g) \times (f,g)}^{\ast} \underline{\pi_{1,2}}^{\ast} D_Y (r) &\wedge \underline{(f,g) \times (f,g)}^{\ast} \underline{\pi_{1,3}}^{\ast} D_Y(\epsilon_f (s)) \wedge \underline{(f,g) \times (f,g)}^{\ast} \underline{\pi_{2,4}}^{\ast} D_Y(\epsilon_g (s))
\end{align*}
from which the desired conclusion follows via the same argument as for (\ref{metricas2pred}).

\end{proof}

We are ready to state the main result of this first half of the paper. The theorem below applies equally well to any of the choices for $\pmetv$ ($\pmetu$, $\pmetl$, or $\pmet$), but then only with the corresponding choices of $E$ ($E_u$, $E_L$, or $E_1$ respectively).

\begin{theorem}
\label{thm:classifier}

There is a correspondence between maps $X \rightarrow |\mathbb{I}|$ with modulus of continuity $\epsilon \in E$ and $\epsilon$-predicates on $X \in \pmetv$:

\begin{enumerate}

\item Given $f: X \rightarrow | \mathbb{I} |$ with modulus of continuity $\epsilon \in E$, $R_f = f^\ast \mct_{\mathbb{I}}$ is an $\epsilon$-predicate on $X$, i.e. is an object of $\Sube X$, and; \label{thm:classifier:fntosub}

\item Given $R \in \Sube X$, the function $f_R: X \rightarrow | \mathbb{I} |$ defined by $f_R (x) = \inf \{r \in \mathbb{I} \mid x \in R(r)\}$ (where the infimum is taken in $\Iop$, i.e. the usual ordering of the unit interval) is a uniformly continuous map $f_R: X \rightarrow | \mathbb{I} |$ with the same modulus $\epsilon$.\label{thm:classifier:subtofn}

\end{enumerate}

These operations are inverse to each other, and they are moreover natural in $X$ in the sense that for each $g: Y \rightarrow X$, we have $R_{f \circ g} = g^\ast R_f$ and $f_{g^\ast R} = f_{R} \circ g$.

\end{theorem}

\begin{proof}

(\ref{thm:classifier:fntosub}) is immediate by Proposition \ref{prop:pbofcontsubobj} and the observation that $\mct_{\mathbb{I}}$ is a $1_{\mathbb{I}}$-predicate.

(\ref{thm:classifier:subtofn}): By Lemma \ref{lem:iff} it suffices to show that $R = (f_R) ^\ast \mct_{\mathbb{I}}$. Now $(f_R)^\ast \mct_{\mathbb{I}}(r) = f_R^{-1}([0,r]) = \{x \in X \mid f_R(x) \leq r\}$. This last set is the set of all points $x \in X$ such that $\inf \{ s \in \mathbb{I} \mid x \in R(s) \} \leq r$.

Then clearly we must have that $R(r) \leq (f_R)^\ast \mct_{\mathbb{I}}(r)$ in $\Sube X$. Now $x \in (f_R)^\ast \mct_{\mathbb{I}}(r)$, i.e. $f_R(x) \leq r$ iff there is some (weakly) decreasing sequence $\{r_i\}$ converging to $r$ such that $x \in R(r_i)$ for all $i$. By assumption we have $R(r) = R(\inf\limits_i r_i) = \bigwedge\limits_i R(r_i)$ so $x \in R(r)$; thus we also have $(f_R)^\ast \mct_{\mathbb{I}}(r) \leq R(r)$ in $\Sube X$.

To show that these operations are inverse to each other, we first show that $f = f_{R_f}$. For each $x \in X$, we have that $f_{R_f} (x) = \inf \{r \in \mathbb{I} \mid x \in R_f(r)\}$. But each $R_f(r) = f^{-1}([0,r])$, so $f_{R_f}(x)$ is the infimum of all $r \in \mathbb{I}$ such that $f(x) \leq r$, which must thus be the original value $f(x)$. Thus $f = f_{R_f}$.

In the other direction, the proof of (\ref{thm:classifier:subtofn}) above already shows that $R = (f_R)^\ast \mct_{\mathbb{I}}$. 

$R_{f \circ g} = g^\ast R_f$ easily follows from the fact that $g^\ast f^\ast = (f \circ g)^\ast$. Lastly, since
\begin{equation*}
R_{f_{g^\ast R}} = g^\ast R = g^\ast (R_{f_R}) = R_{f_R \circ g}
\end{equation*}
(and since $f \mapsto R_f$ is injective) we must have that $f_{g^\ast R} = f_R \circ g$.

\end{proof}

We will later refer to the data $\Omega = | \mathbbm{I} |$ and $\mct_{\mathbbm{I}}$ as a \emph{predicate classifier} for $\pmetv$; see Definition \ref{def:predclass}. The name ``predicate classifier'' (or ``continuous subobject classifier'') is justified by the observation that for any r-geometric category $\mcc$ (such as $\catsets$ or any presheaf category), if instead of $\mathbbm{I}$ we take $\mathbbm{2}$ and make all the obvious resulting simplifications:
\begin{enumerate}[label=$\circ$,ref=$\circ$]

\item For every $X \in \mcc$, define $D_X \in [\Twop, \Sub X]$ by setting $D_X(0)$ to be the diagonal (and $D_X(1)$ is automatically terminal in $\Sub X$);

\item Define $E$ to be the $\mathbbm{2}$-moduloid $\{1_{\mathbbm{2}}\}$;

\end{enumerate}

then a predicate on $X$ is exactly just a subobject of $X$, so an object $\Omega \in \mcc$ and $\mct_{\mathbbm{2}} \in \Sub \Omega$ satisfying the evident analogue of the results of Theorem \ref{thm:classifier} (i.e. for $\mathbbm{2}$ instead of $\mathbb{I}$, and $\Omega$ in place of $| \mathbb{I}|$) exhibit precisely a (regular) subobject classifier for $\mcc$.

\subsection{Some adjoints}

The correspondence of Theorem \ref{thm:classifier} is well-behaved with respect to the relevant logical operations; to describe what we mean by ``relevant logical operations'' we must first establish some facts.

For each $\epsilon$, the property of being an $\epsilon$-predicate is preserved under taking limits (meets):

\begin{proposition}
\label{prop:limprecont}

Let $\epsilon \in E$, and $R_i \in \Sube X$.

Then $\bigwedge\limits_i R_i$ (where the meet is taken in $[\Iop, \Sub X]$) is again an object of $\Sube X$.

\end{proposition}

\begin{proof}

Let $R_i \in \Sube X$, which we consider as objects of $[\Iop, \Sub X]$, and take their meet $\bigwedge\limits_i R_i$ in this latter category. Let $r_j \in \mcv$ such that $\inf\limits_j r_j = r$. We show that
\begin{equation*}
\left ( \bigwedge\limits_i R_i \right ) (r) = \bigwedge\limits_j \left ( ( \bigwedge\limits_i R_i ) (r_j) \right ).
\end{equation*}
Because limits are taken pointwise in functor categories, for each $r \in \mathbb{I}$ we have that
\begin{equation*}
\left ( \bigwedge\limits_i R_i \right ) (r) = \bigwedge\limits_i \left ( R_i (r) \right )
\end{equation*}
where the meet of the right hand side is taken in $\Sub X$. By assumption, we have that $R_i (r) = \bigwedge\limits_j (R_i (r_j))$. Lastly, limits commute with limits, giving us
\begin{equation*}
\left ( \bigwedge\limits_i R_i \right ) (r) = \bigwedge\limits_i \left ( R_i (r) \right ) = \bigwedge\limits_i \left (\bigwedge\limits_j (R_i (r_j)) \right ) = \bigwedge\limits_j \left (\bigwedge\limits_i (R_i (r_j)) \right ) = \bigwedge\limits_j \left ( ( \bigwedge\limits_i R_i ) (r_j) \right ).
\end{equation*}

Now let $r,s \in \mathbb{I}$ be arbitrary. We have that $\left ( \underline{\pi_1}^{\ast} \bigwedge\limits_i R_i (r) \right ) \wedge D_X (s) = \bigwedge\limits_i \left ( \underline{\pi_1}^{\ast} R_i (r) \wedge D_X (s) \right ) \leq \bigwedge\limits_i \underline{\pi_2}^\ast R_i (r + \epsilon (s) ) = \underline{\pi_2}^\ast \bigwedge\limits_i R_i (r+\epsilon(s))$ (recall that (co)limits are taken pointwise in functor categories).

\end{proof}

Therefore we have that the inclusion $\Sube X \hookrightarrow [\Iop, \Sub X]$ preserves meets, and thus has a left adjoint by any of the adjoint functor theorems. Proposition \ref{prop:limprecont} also shows that for any $\epsilon_1 \leq \epsilon_2$, the inclusion $\Subeo X \hookrightarrow \Subet X$ also preserves meets (since meets in either category are computed in $[\Iop, \Sub X]$). Now any left adjoint to an inclusion of a full subcategory into a skeletal (e.g. posetal) category must also be a left inverse for formal reasons. We thus have the following:

\begin{corollary}
\label{cor:leftadjcont}
\item
\begin{enumerate}

\item For each $\epsilon \in E$, we have a left adjoint (and left inverse) $L_{\epsilon}: [\Iop, \Sub X] \rightarrow \Sube X$ to the inclusion $\Sube X \hookrightarrow [\Iop, \Sub X]$.

\item For each $\epsilon_1 \leq \epsilon_2$, we have a left adjoint (and left inverse) $L^{\epsilon_2}_{\epsilon_1}: \Subet X \rightarrow \Subeo X$ to the inclusion $\Subeo X \hookrightarrow \Subet X$.

\item These left adjoints are natural in $\epsilon$, in the following sense:\label{natladjs}

	\begin{enumerate}

	\item $L^{\epsilon}_{\epsilon_1} \circ L_{\epsilon} = L_{\epsilon_1}$

	\item $L^{\epsilon_2}_{\epsilon_1} \circ L^{\epsilon_3}_{\epsilon_2} = L^{\epsilon_3}_{\epsilon_1}$

	\end{enumerate}

\item $L^{\epsilon_2}_{\epsilon_1}: \Subet X \rightarrow \Subeo X$ is equal to $L_{\epsilon_1}: [\Iop, \Sub X] \rightarrow \Subeo X$ restricted along the inclusion $\Subet X \hookrightarrow [\Iop, \Sub X]$.\label{adjsrest}

\end{enumerate}

\end{corollary}

\begin{proof}

We only have left to prove (\ref{natladjs}) and (\ref{adjsrest}).

(\ref{natladjs}) is immediate by the fact that the compositions of inclusions
\begin{align*}
\Subeo X \hookrightarrow & \; \Sube X \hookrightarrow [\Iop, \Sub X]\\
\Subeo X \hookrightarrow & \; \Subet X \hookrightarrow \Subeth X
\end{align*}
are respectively equal to the inclusions
\begin{align*}
\Subeo X \hookrightarrow & \; [\Iop, \Sub X]\\
\Subeo X \hookrightarrow & \; \Subeth X.
\end{align*}

(\ref{adjsrest}) follows from (\ref{natladjs}) and the fact that the left adjoints are left inverses. Denote by $i_{\epsilon_2}: \Subet X \hookrightarrow [\Iop, \Sub X]$ the inclusion. Then $L_{\epsilon_1} i_{\epsilon_2} = L^{\epsilon_2}_{\epsilon_1} L_{\epsilon_2} i_{\epsilon_2} = L^{\epsilon_2}_{\epsilon_1}$.

\end{proof}

\begin{remark}
\label{rmk:explicitadj}

The proof of the adjoint functor theorem actually gives us an explicit description of the left adjoints above. For $\epsilon \in E$, $L_{\epsilon}$ takes any $R \in [\Iop, \Sub X]$ and returns $L_{\epsilon} R = \bigwedge \{ P \in \Sube X \mid R \leq P\}$. In light of Theorem \ref{thm:classifier} above, we see that this is the category theoretic solution to the problem of finding the closest approximation of an arbitrary real-valued function from below by a uniformly continuous function.

\end{remark}

The significance of these left adjoints is that for each $\epsilon \in E$ and each map $f: X \rightarrow Y$ with modulus of continuity $\epsilon_f$, we have a left adjoint $\exists_f$ to the pullback $f^\ast: \Sube Y \rightarrow \Subeef X$, as well as (for certain kinds of $f$) a right adjoint $\forall_f$; these are the ``relevant logical operations'' we mentioned previously.

Note that just as pointwise application of $\underline{f}^\ast: \Sub Y \rightarrow \Sub X$ yielded a functor $\overline{f}^\ast: [\Iop, \Sub Y] \rightarrow [\Iop, \Sub X]$, pointwise application of $\underline{\exists}_f: \Sub X \rightarrow \Sub Y$ (resp. $\underline{\forall}_f$) yields a left (resp. right) adjoint $\overline{\exists}_f: [\Iop, \Sub X] \rightarrow [\Iop, \Sub Y]$ (resp. $\overline{\forall}_f$) to $\overline{f}^\ast$.

\begin{proposition}
\label{prop:infsupadjs}

Let $\epsilon \in E$ and $f: X \rightarrow Y$ with modulus of continuity $\epsilon_f \in E$. Let $f^\ast: \Sube Y \rightarrow \Subeef X$ be as given in Proposition \ref{prop:pbofcontsubobj}.

\begin{enumerate}

\item $\exists_f = L_{\epsilon} \overline{\exists}_f i_{(\epsilon \circ \epsilon_f)}: \Subeef X \rightarrow \Sube Y$ is left adjoint to $f^\ast: \Sube Y \rightarrow \Subeef X$.\label{ladjtopb}

\item Let $f = \pi_X: Y \times X \rightarrow X$ be the projection, so that $\epsilon_f = 1_{\mathbb{I}}$. Also assume that $Y$ is inhabited. We have that $\forall_{\pi_X} = L_{\epsilon} \overline{\forall}_{\pi_X} i_{\epsilon}: \Sube (Y \times X) \rightarrow \Sube X$ is right adjoint to $\pi_X^\ast: \Sube X \rightarrow \Sube (Y \times X)$\label{radjtopb}

\end{enumerate}

where $i_{(\epsilon \circ \epsilon_f)}: \Subeef X \hookrightarrow [\Iop, \Sub X]$ is the inclusion.

\end{proposition}

\begin{proof}

Let $i_{\epsilon}: \Sube Y \hookrightarrow [\Iop, \Sub Y]$ denote the inclusion. The way Proposition \ref{prop:pbofcontsubobj} gave $f^\ast: \Sube Y \rightarrow \Subeef X$ was to show that $\overline{f}^\ast i_{\epsilon}: \Sube Y \rightarrow [\Iop, \Sub X]$ actually lands inside $\Subeef X$, so that $\overline{f}^{\ast} i_{\epsilon} = i_{(\epsilon \circ \epsilon_f)} f^\ast$ for some $f^\ast: \Sube Y \rightarrow \Subeef X$.

Then by Corollary \ref{cor:leftadjcont} we must have that $f^\ast = L_{(\epsilon \circ \epsilon_f)} \overline{f}^\ast i_{\epsilon}$, as well as $i_{(\epsilon \circ \epsilon_f)} L_{(\epsilon \circ \epsilon_f)} \overline{f}^\ast i_{\epsilon} = \overline{f}^\ast i_{\epsilon}$.

To show (\ref{ladjtopb}) we need to show
\begin{equation*}
L_{\epsilon} \overline{\exists}_f i_{(\epsilon \circ \epsilon_f)} R \leq P \Longleftrightarrow R \leq L_{(\epsilon \circ \epsilon_f)} \overline{f}^\ast i_{\epsilon} P
\end{equation*}
for all $R \in \Subeef X$ and $P \in \Sube Y$. Because $i_{\epsilon \circ \epsilon_f}$ is a full inclusion we have
\begin{equation*}
R \leq L_{(\epsilon \circ \epsilon_f)} \overline{f}^\ast i_{\epsilon} P \Longleftrightarrow i_{\epsilon \circ \epsilon_f} R \leq \left ( i_{\epsilon \circ \epsilon_f} L_{(\epsilon \circ \epsilon_f)} \overline{f}^\ast i_{\epsilon} P = \overline{f}^\ast i_{\epsilon} P \right ).
\end{equation*}

By composing adjoints we have that $L_{\epsilon} \overline{\exists}_f$ is left adjoint to $\overline{f}^\ast i_{\epsilon}$, so we have
\begin{equation*}
i_{\epsilon \circ \epsilon_f} R \leq \overline{f}^\ast i_{\epsilon} P \Longleftrightarrow L_{\epsilon} \overline{\exists}_f i_{\epsilon \circ \epsilon_f} R \leq P.
\end{equation*}

To show (\ref{radjtopb}), it suffices to show that $P \leq \forall_{\pi_X} \pi_X^{\ast} P$ for all $P \in \Sube X$, and $\pi_X^{\ast} \forall_{\pi_X} R \leq R$ for all $R \in \Sube (Y \times X)$.

$P \leq \forall_{\pi_X} \pi_X^{\ast} P = L_{\epsilon} \overline{\forall}_{\pi_X} i_{\epsilon} L_{\epsilon} \overline{\pi_X}^{\ast} i_{\epsilon}$ follows immediately by adjointness and the fact that $L_{\epsilon} i_{\epsilon}$ is the identity.

To show that $\pi_X^\ast \forall_{\pi_X} R = L_{\epsilon} \overline{\pi_X}^{\ast} i_{\epsilon} L_{\epsilon} \overline{\forall}_{\pi_X} i_{\epsilon} R \leq R$ for all $R \in \Sube (Y \times X)$, we first show that $L_{\epsilon} \overline{\pi_X}^{\ast} i_{\epsilon} L_{\epsilon} R = L_{\epsilon} \overline{\pi_X}^{\ast} R$ for all $R \in \Sube X$.

$L_{\epsilon} \overline{\pi_X}^{\ast} R \leq L_{\epsilon} \overline{\pi_X}^{\ast} i_{\epsilon} L_{\epsilon} R$ is immediate by adjointness, so it remains to show $L_{\epsilon} \overline{\pi_X}^{\ast} i_{\epsilon} L_{\epsilon} R \leq L_{\epsilon} \overline{\pi_X}^{\ast} R$. By Remark \ref{rmk:explicitadj}, the left hand side is $\bigwedge \{ P \mid \overline{\pi_X}^{\ast} i_{\epsilon} L_{\epsilon} R \leq i_{\epsilon} P \text{ and } P \in \Sube X \}$, while the right hand side is $\bigwedge \{ P \mid \overline{\pi_X}^{\ast} R \leq i_{\epsilon} P \text{ and } P \in \Sube X \}$. Thus it suffices to show that any $P \in \Sube X$ for which $\overline{\pi_X}^{\ast} R \leq i_{\epsilon} P$ also satisfies $\overline{\pi_X}^{\ast} i_{\epsilon} L_{\epsilon} R \leq i_{\epsilon} P$.

Now $\overline{\pi_X}^{\ast} R \leq i_{\epsilon} P$ means that $i_{\epsilon} P = \overline{\pi_X}^{\ast} P^{\prime}$ for some $P^{\prime} \in [\Iop, \Sub X]$. We also have that $\left ( \underline{\pi_1}^{\ast} \underline{\pi_X}^{\ast} P^\prime (r) \right ) (r) \wedge D_{Y \times X} (s) \leq \underline{\pi_2}^{\ast} \underline{\pi_X}^{\ast} P^\prime (r + \epsilon (s) )$. Recall that $D_{X \times Y} (s) = \underline{\pi_{Y \times Y}}^{\ast} D_Y (s) \wedge \underline{\pi_{X \times X}}^{\ast} D_X (s)$, and also that we commutative squares
\begin{equation*}
\begin{tikzcd}
Y \times X \times Y \times X
\arrow{r}{\pi_i}
\arrow{d}{\pi_{X \times X}}
	& Y \times X
	\arrow{d}{\pi_X} \\
X \times X
\arrow{r}{\pi_i}
	& X
\end{tikzcd}
\end{equation*}
for $i = 1,2$, giving us that 
\begin{equation*}
\left ( \underline{\pi_{X \times X}}^{\ast} \underline{\pi_1}^{\ast} P^\prime (r) \right ) (r) \wedge \underline{\pi_{Y \times Y}}^{\ast} D_Y (s) \wedge \underline{\pi_{X \times X}}^{\ast} D_X (s) \leq \underline{\pi_{X \times X}}^{\ast} \underline{\pi_2}^{\ast} P^\prime (r + \epsilon (s) )
\end{equation*}
which is equivalent to
\begin{equation*}
\underline{\pi_{X \times X}}^{\ast} \left ( \underline{\pi_1}^{\ast} P^\prime (r) \wedge D_X (s) \right ) = \left ( \underline{\pi_{X \times X}}^{\ast} \underline{\pi_1}^{\ast} P^\prime (r) \right ) (r) \wedge \underline{\pi_{X \times X}}^{\ast} D_X (s) \leq \underline{\pi_{X \times X}}^{\ast} \underline{\pi_2}^{\ast} P^\prime (r + \epsilon (s) )
\end{equation*}

Now since $Y$ is inhabited, $\underline{\pi_{X \times X}}^{\ast}: \Sub (X \times X) \rightarrow \Sub (Y \times X \times Y \times X)$ is a full embedding for formal reasons, so we finally get
\begin{equation*}
\left ( \underline{\pi_1}^{\ast} P^\prime (r) \right ) \wedge D_X (s) \leq \underline{\pi_2}^{\ast} P^\prime (r + \epsilon (s) ).
\end{equation*}

Moreover we have that for any $r, r_i \in \mathbb{I}$ such that $r = \inf\limits_i r_i$, $\underline{\pi_X}^{\ast} P^\prime (r) = \bigwedge\limits_i \underline{\pi_X}^{\ast} P^\prime (r_i) = \underline{\pi_X}^\prime \bigwedge\limits_i P^\prime (r_i)$. Again by injectivity of $\underline{\pi_X}^\ast: \Sub X \rightarrow \Sub (Y \times X)$ we have that $P^\prime (r) = \bigwedge\limits_i P^\prime (r_i)$.

We have thus verified that $P^\prime$ is an $\epsilon$-predicate on $X$, and therefore is of the form $i_\epsilon P^{\prime\prime}$ for some $P^{\prime\prime} \in \Sube X$. Then $i_{\epsilon} P = \overline{\pi_X}^{\ast} i_{\epsilon} P^{\prime\prime}$, so that $\overline{\pi_X}^{\ast} R \leq i_{\epsilon} P = \overline{\pi_X}^{\ast} i_{\epsilon} P^{\prime\prime}$. By injectivity of $\overline{\pi_X}^{\ast}$ (since $\underline{\pi_X}^{\ast}$ is injective), we have that $R \leq i_{\epsilon}P^{\prime\prime}$ which implies that $i_{\epsilon} L_{\epsilon} R \leq i_{\epsilon} P^{\prime\prime}$ by adjointness, and therefore we conclude that $\overline{\pi_X}^{\ast} i_{\epsilon} L_{\epsilon} R \leq \overline{\pi_X}^{\ast} i_{\epsilon} P^{\prime\prime} =  i_{\epsilon} P$.

Therefore we have that $L_{\epsilon} \overline{\pi_X}^{\ast} i_{\epsilon} L_{\epsilon} \overline{\forall}_{\pi_X} i_{\epsilon} R = L_{\epsilon} \overline{\pi_X}^{\ast} \overline{\forall}_{\pi_X} i_{\epsilon} R \leq R$ for all $R \in \Sube (Y \times X)$, completing the proof.

\end{proof}

The ``relevant logical operations'' mentioned earlier refers to the adjoints $\exists_{\pi_X}$ (resp. $\forall_{\pi_X}$) to the pullback $\pi_X^\ast$, for the projection $\pi_X: Y \times X \rightarrow X$. A standard paradigm in categorical semantics is that these adjoints should correspond to ``existential (resp. universal) quantification over $Y$'' in some appropriate sense; we show that in our case, these adjoints - which we arrived at purely formally - correspond precisely to those quantification operations prescribed by continuous logic, in the following sense.

\begin{proposition}
\label{prop:adjsarequants}

For each $X, Y \in \pmetv$, and each $R \in \Sube (Y \times X)$ with $\pi_X: Y \times X \rightarrow X$, the correspondence $R \mapsto f_R$ of Theorem \ref{thm:classifier} gives correspondences

\begin{enumerate}

\item $\exists_{\pi_X} R \in \Sube X \mapsto \inf\limits_{y \in Y} f_R (y, -): X \rightarrow |\mathbb{I}|$\label{existspf}

\item If $Y$ is inhabited: $\forall_{\pi_X} R \in \Sube X \mapsto \sup\limits_{y \in Y} f_R (y, -): X \rightarrow |\mathbb{I}|$.\label{forallpf}

\end{enumerate}

\end{proposition}

\begin{proof}

(\ref{existspf}): Recall that $\exists_{\pi_X} R = L_{\epsilon} \overline{\exists}_{\pi_X} i_{\epsilon} R$, or in other words $\exists_{\pi_X} R$ is the smallest $\epsilon$-predicate on $X$ containing $\overline{\exists}_{\pi_X} i_{\epsilon} R$. Denote $\varphi = \inf\limits_{y \in Y} f_R (y, -)$. We will show that $R_{\varphi} = \varphi^{\ast} \mct_{\mathbb{I}}$ must necessarily be the smallest $\epsilon$-predicate containing $\overline{\exists}_{\pi_X} i_{\epsilon} R$, which will imply that $R_\varphi = \exists_{\pi_X} R$ and so $f_{\exists_{\pi_X} R} = f_{R_{\varphi}} = \varphi$.

First notice that $x \in \overline{\exists}_{\pi_X} i_{\epsilon} R(r)$ iff there is some $y \in Y$ such that $(y,x) \in i_{\epsilon}R(r)$. On the other hand, $x \in \varphi^\ast \mct_{\mathbb{I}} (r)$ iff $\inf\limits_{y \in Y} f_R(y, x ) = \inf\limits_{y\in Y} \inf \{ s \mid (y, x) \in i_{\epsilon} R(s) \} \leq r$. Clearly $\overline{\exists}_{\pi_X} i_{\epsilon} R(r) \leq \varphi^\ast \mct_{\mathbb{I}} (r)$.

For the other direction, assume $x \in \varphi^\ast \mct_{\mathbb{I}} (r)$. That is, there are sequences $y_i \in Y$ and $r_{i,j}, r_i \in \mathbb{I}$ such that $(y_i, x) \in i_{\epsilon} R(r_{i,j})$ for all $i,j$ and $r_{i,j}$ is weakly decreasing and converges to $r_i$ for each $i$, with $r_i$ weakly decreasing and converging to $r$. Now because $R$ is an $\epsilon$-predicate, we must have that $(y_i, x) \in i_{\epsilon} R(r_i)$ for all $i$. Then we have that $x \in \overline{\exists}_{\pi_X} i_{\epsilon} R (r_i)$ for all $i$. Then if $P$ is any $\epsilon$-predicate containing $\overline{\exists}_{\pi_X} i_{\epsilon} R$ then $x \in P(r)$. Therefore $\varphi^{\ast} \mct_{\mathbb{I}}$ must be the smallest $\epsilon$-predicate containing $\overline{\exists}_{\pi_X}$, and $f_{\exists_{\pi_X}} = \inf\limits_{y \in Y} f_R(y, -)$.

(\ref{forallpf}): $x \in \overline{\forall}_{\pi_X} i_{\epsilon} R(r)$ iff $(y, x) \in i_{\epsilon} R(r)$ for all $y \in Y$. On the other hand, let $\psi = \sup\limits_{y \in Y} f_R (y, -)$, so that $x \in \psi^{\ast} \mct_{\mathbb{I}}(r)$ iff $\sup_{y \in Y} f_R (y, x) = \sup_{y \in Y} \inf \{s \mid (y, x) \in i_{\epsilon} R(s) \} \leq r$. Clearly $\overline{\forall}_{\pi_X} i_{\epsilon} R \leq \psi^{\ast} \mct_{\mathbb{I}}$.

In the other direction, if $x \in \psi^{\ast} \mct_{\mathbb{I}}(r)$ then for every $y \in Y$ there is some sequence $r_i$ which is weakly decreasing and converging to $r$, with $(y,x) \in i_{\epsilon} R(r_i)$ for all $i$. Again because $R$ is an $\epsilon$-predicate, we must have that $(y,x) \in i_{\epsilon} R(r)$. Therefore $\psi^{\ast} \mct_{\mathbb{I}}(r) \leq \overline{\forall}_{\pi_X} i_{\epsilon} R$, thus $\overline{\forall}_{\pi_X} i_{\epsilon} R$ was an $\epsilon$-predicate to begin with and is equal to $\psi^{\ast} \mct_{\mathbb{I}} (r)$, so that $f_{\forall_{\pi_X}} = \sup\limits_{y \in Y} f_R(y, -)$.

\end{proof}

\section{Presheaves of metric spaces}
\label{sec:presheaves}

Thus far we have isolated features of the category $\pmetv$ giving rise to much of the structure (pertaining e.g. to continuous logic) present in $\pmetv$. The upshot is that we may identify the same structures in more general categories. A natural example to consider is categories of presheaves of metric spaces; not only do they mirror the generalization from sets to presheaves of sets, they also find application in highly practical contexts \cite{robinson2}, \cite{robinson1} pertaining to data integration.

We will see that the material developed in this section exhibits categories of presheaves of metric spaces as having much of the same essential structure as $\pmetv$, to be described in a precise sense.

\subsection{Setup}

We first set up some general categorical machinery that will allow us to apply our preceding analysis of the category $\pmetv$ to the category of presheaves.

Recall Definition \ref{def:rreg} of an r-geometric category as one having suitable structure on its subobject posets.

The below is essentially a restatement of Proposition \ref{prop:metricprops} and Lemma \ref{lem:continuity} together; we take the features of $\pmetv$ that Proposition \ref{prop:metricprops} and Lemma \ref{lem:continuity} guarantee and turn them into a definition which we may ask any r-geometric category to satisfy.

\begin{definition}
\label{def:metrization}

Let $\mcc$ be an r-geometric category, and $E$ an $\mathbb{I}$-moduloid.

\begin{enumerate}

\item A \emph{metrization} of $\mcc$ is a choice

\begin{enumerate}[label=$\circ$, ref=$\circ$]

\item for each $X, Y \in \mcc$, of product $X \times Y \in \mcc$, well-behaved in the evident sense with respect to symmetry and associativity, and;

\item for each $X \in \mcc$, of $D_X \in [\Iop, \Sub (X \times X)]$ satisfying the following:

\end{enumerate}

	\begin{enumerate}

\item $D_X(0)$ contains the diagonal;

\item The functor $[\Iop, \Sub (X \times X)] \xrightarrow{\cong} [\Iop, \Sub (X \times X)]$ induced by the symmetry isomorphism $X \times X \xrightarrow{\cong} X \times X$ interchanging the factors takes $D_X$ to itself;

\item Let $\pi_{i,j}: (X \times X \times X) \rightarrow X$ denote the projection onto the $i^{\text{th}}$ and $j^{\text{th}}$ factors, respectively. Then $\underline{\pi_{i,j}}^{\ast} D_X(r) \wedge \underline{\pi_{j,k}}^{\ast} D_X(s) \leq \underline{\pi_{i,k}}^{\ast} D_X (r + s)$ for every $r, s \in \mathbb{I}$.

\item If $r = \inf\limits_i r_i$ for some $r, r_i \in \mathbb{I}$, then $D_X(r) = \bigwedge\limits_i D_X (r_i)$;

\item Let $\pi_{X \times X}: (X \times Y \times X \times Y) \rightarrow (X \times X)$ and $\pi_{Y \times Y}: (X \times Y \times X \times Y) \rightarrow (Y \times Y)$ denote the projections preserving the ordering of the factors. Then $D_{X \times Y}(r) = \underline{\pi_{X \times X}}^{\ast} D_X(r) \wedge \underline{\pi_{Y \times Y}}^{\ast} D_Y(r)$ for all $r \in \mathbb{I}$.\label{prodmetric2}

	\end{enumerate}

\item Given a metrization of $\mcc$, we say that it is \emph{compatible with $E$} when for each morphism $f: X \rightarrow Y$ in $\mcc$, there is some $\epsilon \in E$ such that for every $r \in \mathbb{I}$, we have that
\begin{equation*}
D_X (r) \leq \underline{f \times f}^{\ast} D_Y ( \epsilon(r) )
\end{equation*}
and in this case we say that $f: X \rightarrow Y$ is \emph{continuous with respect to $\epsilon$}, or that \emph{$f$ has modulus $\epsilon$}.
\end{enumerate}

\end{definition}

Recall from Example \ref{ex:moduloids} the various $\mathbbm{I}$-moduloids $E_u, E_L, E_1$. From Lemma \ref{lem:continuity} in Section \ref{sec:indexedsubs}, we can see that
\begin{enumerate}

\item $\pmetu$ has a metrization compatible with $E_u$;

\item $\pmetl$ has a metrization compatible with $E_L$;

\item $\pmet$ has a metrization compatible with $E_1$.

\end{enumerate}

The point is that as soon as a category $\mcc$ has a metrization compatible with $E \subseteq \Endo(\mathbb{I})$ as in Definition \ref{def:metrization}, we immediately get results about $\mcc$ that make precise the idea that it behaves like a category of metric spaces. Specifically:

\begin{proposition}
\label{prop:metrizationcors}

Let $E$ be an $\mathbb{I}$-moduloid and $\mcc$ have a metrization compatible with $E$. Then the following holds:

\begin{enumerate}

\item For $X \xrightarrow{f} Y \xrightarrow{g} Z$ a composition of morphisms in $\mcc$ such that $f$ and $g$ have moduli $\epsilon_f$ and $\epsilon_g$ respectively, we have that $gf: X \rightarrow Z$ has modulus $\epsilon_g \circ \epsilon_f$.\label{compmodulus2}

\item For any $X, Y \in \mcc$, $\pi_X: X \times Y \rightarrow X$ satisfies
\begin{equation*}
D_{X \times Y} (r) \leq \underline{\pi_X \times \pi_X}^{\ast} D_X (r)
\end{equation*}
for all $r \in \mathbb{R}$.\label{projmodulus2}

\item Each pair of maps $f: X \rightarrow Y$, $g: X \rightarrow Z$ (with moduli $\epsilon_f$ and $\epsilon_g$ respectively) canonically corresponds to the obvious map $(f,g): X \rightarrow (Y \times Z)$, with modulus $\epsilon_{(f,g)} = \max (\epsilon_f, \epsilon_g)$.

In the other direction, if $\epsilon_{(f,g)}$ is a modulus for $(f,g)$ then it is also a modulus for both $f$ and $g$.\label{prodmodulus2}

\end{enumerate}

\end{proposition}

\begin{proof}

The same proof as for Proposition \ref{prop:formalpfs} goes through unchanged.

\end{proof}

We repeat the definition of an $\epsilon$-predicate, now in the general case.

\begin{definition}
\label{def:econtsubobj2}

Let $E$ be an $\mathbb{I}$-moduloid and $\mcc$ have a metrization compatible with $E$.

Let $X \in \mcc$, and let $\epsilon: \mathbb{I} \rightarrow \mathbb{I}$ be an object of $E$.

We call $R \in [\Iop, \Sub X]$ an \emph{$\epsilon$-predicate on $X$} when:

\begin{enumerate}

\item Given $r, r_i \in \mathbb{I}$ such that $r = \inf\limits_i r_i$, $R(r) = \bigwedge\limits_i R(r_i)$, and;

\item For each $r,s \in \mathbb{I}$,
\begin{equation*}
\left ( \underline{\pi_1}^{\ast} R (r) \right ) \wedge D_X (s) \leq \underline{\pi_2}^{\ast} R (r + \epsilon(s))
\end{equation*}

\end{enumerate}

We denote by $\Sube X$ the full subcategory of $[\Iop, \Sub X]$ on the $\epsilon$-predicates.

Let us call $R \in [\Iop, \Sub X]$ a \emph{predicate on $X$} when there exists some $\epsilon \in E$ for which $R$ is an $\epsilon$-predicate.

\end{definition}

Again as before, for $\epsilon_1 \leq \epsilon_2$ any $\epsilon_1$-predicate is also an $\epsilon_2$-predicate, so we have a full inclusion $\Subeo X \hookrightarrow \Subet X$.

$\epsilon$-predicates in this more general context exhibit the same nice properties as they did in the category $\pmetv$, as we now detail below.

\begin{proposition}
\label{prop:freeresults}

Let $E$ be an $\mathbb{I}$-moduloid and $\mcc$ have a metrization compatible with $E$.

\begin{enumerate}

\item Given $f: X \rightarrow Y$ a morphism in $\mcc$ with modulus of continuity $\epsilon_f$, and given $R \in [\Iop, \Sub Y]$ which is an $\epsilon$-predicate, we have that $\overline{f}^\ast R \in [\Iop, \Sub X]$ is an $(\epsilon \circ \epsilon_f)$-predicate.

In particular, $\overline{f}^\ast: [\Iop, \Sub Y] \rightarrow [\Iop, \Sub X]$ descends to a functor $f^\ast: \Sube Y \rightarrow \Subeef X$ for which $i_{\epsilon \circ \epsilon_f} f^{\ast} = \overline{f}^{\ast} i_{\epsilon}$. \label{pbpreds2}

\item Let $\epsilon \in E$, $X \in \mcc$, and $R_i \in \Sube X$.

Then $\bigwedge\limits_i R_i$ (where the meet is taken in $[\Iop, \Sub X]$) is again an object of $\Sube X$.\label{meetpreserve2}

\item Let $X \in \mcc$.

\begin{enumerate}

\item For each $\epsilon \in E$, we have a left adjoint (and left inverse) $L_{\epsilon}: [\Iop, \Sub X] \rightarrow \Sube X$ to the inclusion $\Sube X \hookrightarrow [\Iop, \Sub X]$.

\item For each $\epsilon_1 \leq \epsilon_2$, we have a left adjoint (and left inverse) $L^{\epsilon_2}_{\epsilon_1}: \Subet X \rightarrow \Subeo X$ to the inclusion $\Subeo X \hookrightarrow \Subet X$.

\item These left adjoints are natural in $\epsilon$, in the following sense:\label{natladjs2}

	\begin{enumerate}

	\item $L^{\epsilon}_{\epsilon_1} \circ L_{\epsilon} = L_{\epsilon_1}$

	\item $L^{\epsilon_2}_{\epsilon_1} \circ L^{\epsilon_3}_{\epsilon_2} = L^{\epsilon_3}_{\epsilon_1}$

	\end{enumerate}

\item $L^{\epsilon_2}_{\epsilon_1}: \Subet X \rightarrow \Subeo X$ is equal to $L_{\epsilon_1}: [\Iop, \Sub X] \rightarrow \Subeo X$ restricted along the inclusion $\Subet X \hookrightarrow [\Iop, \Sub X]$.\label{adjsrest2}

\end{enumerate}\label{sheafification2}

\item Let $\epsilon \in E$ and $f: X \rightarrow Y$ a morphism in $\mcc$ with modulus of continuity $\epsilon_f \in E$. Let $f^\ast: \Sube Y \rightarrow \Subeef X$ be as given in (\ref{pbpreds2}) above.

\begin{enumerate}

\item $\exists_f = L_{\epsilon} \overline{\exists}_f i_{(\epsilon \circ \epsilon_f)}: \Subeef X \rightarrow \Sube Y$ is left adjoint to $f^\ast: \Sube Y \rightarrow \Subeef X$.\label{ladjtopb2}

\item Let $f = \pi_X: Y \times X \rightarrow X$ be the projection, so that $\epsilon_f = 1_{\mathbb{I}}$.

We have that $\forall_{\pi_X} = L_{\epsilon} \overline{\forall}_{\pi_X} i_{\epsilon}: \Sube (Y \times X) \rightarrow \Sube X$ is right adjoint to $\pi_X^\ast: \Sube X \rightarrow \Sube (Y \times X)$\label{radjtopb2}

\end{enumerate}

where $i_{(\epsilon \circ \epsilon_f)}: \Subeef X \hookrightarrow [\Iop, \Sub X]$ is the inclusion.\label{quantadjoints2}

\end{enumerate}

\end{proposition}

\begin{proof}

(\ref{pbpreds2}): Same proof as for Proposition \ref{prop:pbofcontsubobj}.

(\ref{meetpreserve2}): Same proof as for Proposition \ref{prop:limprecont}.

(\ref{sheafification2}): Same proof as for Corollary \ref{cor:leftadjcont}.

(\ref{quantadjoints2}): Same proof as for Proposition \ref{prop:infsupadjs}.

\end{proof}

Finally, for $\mcc$ equipped with a metrization compatible with an $\mathbb{I}$-moduloid $E$, we may make sense of the notion of a ``continuous subobject classifier'' as follows:

\begin{definition}
\label{def:predclass}

Let $E$ be an $\mathbb{I}$-moduloid and $\mcc$ have a metrization compatible with $E$.

A \emph{predicate classifier} is given by an object $\Omega \in \mcc$ (and its associated $D_{\Omega} \in [\Iop, \Sub (\Omega \times \Omega) ]$) along with a $1_{\mathbb{I}}$-predicate $\mct_{\mathbb{I}} \in \Suboi \Omega$ such that:

For any $R \in \Sube X$ where $\epsilon \in E$, there is a unique $f: X \rightarrow \Omega$ such that $R = f^{\ast} \mct_{\mathbb{I}}$, and moreover this $f$ has modulus of continuity $\epsilon$.

\end{definition}

With this definition, Theorem \ref{thm:classifier} is just the statement that $\pmetv$ has a predicate classifier.

Given that the framework laid out by Proposition \ref{prop:metrizationcors}, Definition \ref{def:econtsubobj2}, Proposition \ref{prop:freeresults}, and Definition \ref{def:predclass} wholly depends on the category $\mcc$ having a metrization (compatible with some $E$) as in Definition \ref{def:metrization}, it is appropriate to ask whether there are any interesting examples of such categories apart from the obvious examples $\pmet$, $\pmetl$, and $\pmetu$.

We will see over the remainder of this section that categories of presheaves of metric spaces provide an affirmative answer to this question, and furthermore that each such category even exhibits a predicate classifier as described in Definition \ref{def:predclass}.

\subsection{The category of presheaves}

For technical ease, we impose some conditions on the kind of presheaves we consider. Recall that $\pmet$ refers to the category of diameter $\leq 1$ pseudometric spaces and $1$-Lipschitz maps between them.

\begin{definition}
\label{def:shortshvs}

Let $\mcc$ be a small category. We call a functor $F: \Cop \rightarrow \pmet$ a \emph{metric presheaf on $\mcc$}, and the functor category $[\Cop, \pmet]$ the \emph{category of metric presheaves on $\mcc$}, which we also denote by $\Psh(\mcc)$.

\end{definition}

\begin{lemma}
\label{lem:pshlims}

$\Psh(\mcc)$ has finite limits.

\end{lemma}

\begin{proof}

$\pmet$ has finite limits, and so we can take limits in $\Psh(\mcc)$ pointwise.

\end{proof}

Now isomorphisms in $\pmet$ are isometries, and for every parallel pair of morphisms their equalizer can be chosen to be an isometric embedding, so we must have that in fact every equalizer (and thus every regular monomorphism) is an isometric embedding.

We have the following convenient description of regular monomorphisms in $\Psh(\mcc)$.

\begin{lemma}
\label{lem:pshregmono}

Let $F, G \in \Psh(\mcc)$. A natural transformation $\mu: F \rightarrow G$ is a regular monomorphism in $\Psh(\mcc)$ if and only if each component $\mu_a: Fa \rightarrow Ga$ is a regular monomorphism in $\pmet$ for each object $a \in \mcc$.

\end{lemma}

\begin{proof}

The forward direction is clear, since limits are taken pointwise in functor categories. We show that given a natural transformation $\mu: F \rightarrow G$ such that for every $a \in \mcc$ the component $\mu_a: Fa \rightarrow Ga$ is a regular mono (and thus an isometric embedding), $\mu$ is in fact an equalizer
\begin{tikzcd}
F
\arrow{r}{\mu}
	& G
	\arrow[shift left]{r}{\nu}
	\arrow[shift right]{r}[swap]{\nu^\prime}
		& H
\end{tikzcd}
for some parallel natural transformations $\nu, \nu^\prime: G \rightarrow H$ (for some $H$).

Define the values of $H \in \Psh(\mcc)$ on each $a \in \mcc$ to be the set of sieves on $a$ equipped with the indiscrete metric (every distance is $0$), and for each morphism $f: b \rightarrow a$ in $\mcc$, we define $Hf: Ha \rightarrow Hb$ to be the evident pullback action on sieves. That is, given $S \in Ha$, $Hf (S) \in Hb$ is the sieve
\begin{equation*}
Hf (S) = f^\ast S = \{h: b^\prime \rightarrow b \mid fh \in S\}
\end{equation*}
and $Hf$ is clearly $1$-Lipschitz since every distance is $0$.

Since each component $\mu_a: Fa \rightarrow Ga$ is an isometric embedding, we may equivalently think of each $Fa$ as a subspace of $Ga$, with $\mu_a$ being the inclusion. We define $\nu: G \rightarrow H$ componentwise as follows. $\nu_a: Ga \rightarrow Ha$ is the map that sends each point $x \in Ga$ to the sieve
\begin{equation*}
\nu_a (x) = \{f: b \rightarrow a \mid Gf (x) \in Fb \}.
\end{equation*}
Now given $x \in Fa \subseteq Ga$ and $f: b \rightarrow a$ in $\mcc$, we verify that $Hf \circ \nu_a (x) = \nu_b \circ Gf (x)$.

$\nu_b ( Gf (x) )$ is the sieve on $b$ consisting of all those morphisms $h: b^\prime \rightarrow b$ such that $Gh (Gf (x)) = G(fh)(x) \in Fb \subseteq Gb$. But this is exactly $Hf (\nu_a (x) )$, so $\nu: G \rightarrow H$ is a natural transformation.

We now define $\nu^\prime: G \rightarrow H$ componentwise as follows. For each $a \in \mcc$ and for each $x \in Ga$, $\nu^\prime_a (x)$ is the maximal sieve on $a$. This is easily verified to be a natural transformation (pulling back the maximal sieve on $a$ across $f: b \rightarrow a$ yields the maximal sieve on $b$).

We take an equalizer of $\nu$ and $\nu^\prime$, call it $\mu^\prime: F^\prime \rightarrow G$. Each component $\mu^\prime_a: F^\prime a \rightarrow Ga$ is an isometric embedding, and is an equalizer of $\nu_a$ and $\nu^\prime_a$. But this means that $\mu_a$ and $\mu^\prime_a$ are isometric embeddings with the same image for each $a \in \mcc$, since a sieve is maximal iff it contains the identity morphism. It is easy to see that this forces $F$ and $F^\prime$ to be naturally isomorphic, so $\mu$ is an equalizer of $\nu$ and $\nu^\prime$, as desired.

\end{proof}

By our previous discussion, regular monos in $\Psh(\mcc)$ are thus natural transformations whose components are isometric embeddings. We formalize this discussion as the following lemma:

\begin{lemma}
\label{lem:whataresubs}

Given a functor $F: \Cop \rightarrow \pmet$, a subobject $\mu \in \Sub F$ uniquely determines the following data:
\begin{enumerate}

\item For each object $a \in \mcc$, a subspace $\mu(a) \subseteq Fa$ such that;

\item For each morphism $f: a \rightarrow b$ in $\mcc$, the map $Ff: Fb \rightarrow Fa$ restricts to a map $\mu(f) = Ff: \mu(b) \rightarrow \mu(a)$.

\end{enumerate}

Conversely, the above data uniquely determines a regular monomorphism with codomain $F$, and therefore a subobject $\mu \in \Sub F$.

\end{lemma}

The above lemma enables us to think of subobjects of $F$ interchangeably with ``compatible'' assignments of subspaces of $Fa$ for each $a \in \mcc$, i.e. natural transformations into $F$ with each component an isometric embedding. We henceforth reserve the right to abuse notation by speaking of subobjects of $F$ as if they were functors whose values on each object $a \in \mcc$ are subspaces of $Fa$.

For any $F \in \Psh(\mcc)$, we know what the poset structure of $\Sub F$ is, since we know what regular monomorphisms look like. We give a useful equivalent description in the following lemma.

\begin{lemma}
\label{lem:sameposet}

Given $\mu, \nu \in \Sub F$, we have $\mu \leq \nu$ iff for all $a \in \mcc$, $\mu(a) \leq \nu(a)$ in $\Sub Fa$.

\end{lemma}

\begin{proof}

The ``only if'' direction is clear. We show that $\mu(a) \leq \nu(a)$ for all $a \in \mcc$ implies $\mu \leq \nu$. Having $\mu, \nu \in \Sub F$ and $\mu(a) \leq \nu(a)$ for all $a \in \mcc$ gives us each arrow of the following diagram
\begin{equation*}
\begin{tikzcd}
\mu(b)
\arrow{r}{Ff}
\arrow[hookrightarrow]{d}
	& \mu(a)
	\arrow[hookrightarrow]{d} \\
\nu(b)
\arrow{r}{Ff}
	& \nu(a)
\end{tikzcd}
\end{equation*}
which we a posteriori conclude is commutative. This gives us $\mu \leq \nu$.

\end{proof}

\begin{proposition}
\label{prop:pshreg}
\item
\begin{enumerate}

\item $\Psh(\mcc)$ has r-images.\label{pshims}

\item These r-images are preserved under pullback.\label{pshimspb}

\item Regular monos in $\Psh(\mcc)$ are closed under composition.\label{pshregmonoscomp}

\end{enumerate}

Therefore, $\Psh(\mcc)$ is an r-regular category.

\end{proposition}

\begin{proof}
\item
(\ref{pshims}): $\pmet$ itself has $r$-images, with every morphism $f: X \rightarrow Y$ in $\pmet$ factoring as $f = ie$ for some epi
\begin{tikzcd}
X
\arrow[twoheadrightarrow]{r}{e}
	&A
\end{tikzcd}
and a regular mono
\begin{tikzcd}
A
\arrow[tail]{r}{i}
	& Y
\end{tikzcd}
serving as the r-image.

Now given $F, G \in \Psh(\mcc)$ and $\varphi: F \rightarrow G$, for each $a \in \mcc$ we can factorize $\varphi_a: Fa \rightarrow Ga$ as $\varphi_a = e_a i_a$ for some epi
\begin{tikzcd}
Fa
\arrow[twoheadrightarrow]{r}{e_a}
	&Da
\end{tikzcd}
and a regular mono
\begin{tikzcd}
Da
\arrow[tail]{r}{i_a}
	& Ga
\end{tikzcd}
serving as the r-image.

This determines the values of a hypothetical functor $D: \Cop \rightarrow \pmet$ on the objects of $\mcc$; we show that this extends to an honest functor $D: \Cop \rightarrow \pmet$ so that there are natural transformations
\begin{tikzcd}
F
\arrow[twoheadrightarrow]{r}{e}
	& D
\end{tikzcd} and
\begin{tikzcd}
D
\arrow[tail]{r}{i}
	& G
\end{tikzcd}
with components $e_a$ and $i_a$ respectively for each $a \in \mcc$.

We specify the values of $D: \Cop \rightarrow \pmet$ on morphisms $f: a \rightarrow b$ of $\mcc$ by the following diagram:
\begin{equation*}
\begin{tikzcd}
Fa
\arrow{rr}[pos=0.4]{\varphi_a}
\arrow[twoheadrightarrow]{dr}{e_a}
	&
		& Ga
		\arrow[shift left]{r}{h}
		\arrow[shift right]{r}[swap]{k}
			& \cdot \\
	& Da
	\arrow[tail]{ur}{i_a}
		&
			& \\
Fb
\arrow{uu}{Ff}
\arrow{rr}[pos=0.3]{\varphi_b}
\arrow[twoheadrightarrow]{dr}[swap]{e_b}
	&
		& Gb
		\arrow{uu}{Gf}
			& \\
	& Db
	\arrow[dotted, crossing over]{uu}[swap, pos=0.75]{\exists ! \; Df} 
	\arrow[tail]{ur}[swap]{i_b}
\end{tikzcd}
\end{equation*}
where $h$ and $k$ were chosen to be a pair of parallel morphisms for which $i_a$ is an equalizer, and we argue for the unique existence of the dotted morphism (which we suggestively call $Df$) as follows.

The unique existence of $Df: Db \rightarrow Da$ is guaranteed if we show that $h \circ Gf \circ i_b = k \circ Gf \circ i_b$. Since $e_b$ is epi, it suffices to show that $h \circ Gf \circ i_b \circ e_b = k \circ Gf \circ i_b \circ e_b$. But we have that
\begin{equation*}
Gf \circ i_b \circ e_b
= Gf \circ \varphi_b
= \varphi_a \circ Ff
= i_a \circ e_a \circ Ff
\end{equation*}
so $h \circ Gf \circ i_b \circ e_b = k \circ Gf \circ i_b \circ e_b$ is true.

The uniqueness of $Df$ for each $f: a \rightarrow b$ guarantees functoriality, so therefore we have specified the data of a functor $D: \Cop \rightarrow \pmet$. Moreover, we have that $i_a \circ Df = Gf \circ i_b$ for each $f: a \rightarrow b$, so the components
\begin{tikzcd}
Da
\arrow[tail]{r}{i_a}
	& Ga
\end{tikzcd}
for each $a \in \mcc$ assemble into a natural transformation
\begin{tikzcd}
D
\arrow[tail]{r}{i}
	& G
\end{tikzcd}.

To see that the components
\begin{tikzcd}
Fa
\arrow[twoheadrightarrow]{r}{e_a}
	& Da
\end{tikzcd}
for each $a \in \mcc$ also assemble into a natural transformation
\begin{tikzcd}
F
\arrow[twoheadrightarrow]{r}{e}
	& D
\end{tikzcd}, we note that by construction (and by naturality of $\varphi$) we have $i_a \circ e_a \circ Ff = i_a \circ Df \circ e_b$, so that $e_a \circ Ff = Df \circ e_b$ since $i_a$ is monic.

We have thus factored $\varphi: F \rightarrow G$ into $\varphi = i e$ where
\begin{tikzcd}
F
\arrow{r}{e}
	& D
\end{tikzcd} is pointwise epi and
\begin{tikzcd}
D
\arrow{r}{i}
	& G
\end{tikzcd} is pointwise an r-image (in particular a regular mono) in $\pmet$. We now show that in this factorization, $i$ is actually an r-image in $\Psh(\mcc)$.

For any other factorization $\varphi = me^\prime$ where
\begin{tikzcd}
D^\prime
\arrow[tail]{r}
	& G
\end{tikzcd}
is a regular mono (and therefore pointwise a regular mono), we have the following diagram:
\begin{equation*}
\begin{tikzcd}[row sep = 3em, column sep = 5em]
	& D^\prime a
	\arrow[tail]{drr}{m_a}
	\arrow[leftarrow]{ddd}[pos=0.33]{D^\prime f}
		&
			& \\
Fa
\arrow{ur}{e^\prime_a}
\arrow[twoheadrightarrow, crossing over]{drr}[pos=0.67]{e_a}
	&
		&
			& Ga \\
	&
		& Da
		\arrow[tail]{ur}{i_a}
		\arrow[tail, dotted]{uul}[swap]{\exists ! \; \psi_a}
			& \\
	& D^\prime b
	\arrow[tail]{drr}[pos=.67]{m_b}
		&
			& \\
Fb
\arrow{uuu}{Ff}
\arrow{ur}{e^\prime_b}
\arrow[twoheadrightarrow]{drr}{e_b}
	&
		&
			& Gb
			\arrow{uuu}{Gf} \\
	&
		& Db
		\arrow[tail]{ur}{i_b}
		\arrow[crossing over]{uuu}[pos=0.67]{Df}
		\arrow[tail, dotted]{uul}{\exists ! \; \psi_b}
			&
\end{tikzcd}
\end{equation*}
where we have unique existence of the dotted morphisms because $i_a$ and $i_b$ are r-images, and where the solid morphisms all commute and the top and bottom faces (including the dotted morphisms) individually commute. To check that the morphisms
\begin{tikzcd}
Da
\arrow[tail]{r}{\psi_a}
	& D^\prime a
\end{tikzcd}
for each $a \in \mcc$ assemble into a natural transformation
\begin{tikzcd}
D
\arrow[tail]{r}{\psi}
	& D^\prime
\end{tikzcd}, we check that $\psi_a \circ Df = D^\prime f \circ \psi_b$, which is equivalent to $m_a \circ \psi_a \circ Df = m_a \circ D^\prime f \circ \psi_b$ since $m_a$ is monic. But this last equality is true since
\begin{equation*}
m_a \circ \psi_a \circ Df = i_a \circ D_f = Gf \circ i_b = Gf \circ m_b \circ \psi_b = m_a \circ D^\prime f \circ \psi_b.
\end{equation*}
The resulting natural transformation
\begin{tikzcd}
D
\arrow[tail]{r}{\psi}
	& D^\prime
\end{tikzcd}
is both a regular mono and unique for formal reasons.

Also, for the same reasons, r-image factorizations are unique up to isomorphism so that all r-image factorizations arise in the way described above, i.e. as pointwise $r$-images following pointwise epis.

(\ref{pshimspb}): Let $\varphi: F \rightarrow G$ have an r-image factorization $\varphi = ie$, and let $\psi: H \rightarrow G$ be some natural transformation. We want to show that when we pull back across $\psi$ as in the diagram below
\begin{equation*}
\begin{tikzcd}
\psi^\ast F
\arrow{d}{\psi^\ast e}
\arrow{r}
	& F
	\arrow{d}{e} \\
\cdot
\arrow{d}{\psi^\ast i}
\arrow{r}
	& \cdot
	\arrow{d}{i} \\
H
\arrow{r}{\psi}
	& G
\end{tikzcd}
\end{equation*}
we have that $\psi^\ast \varphi = (\psi^\ast i) \circ (\psi^\ast e)$ is an r-image factorization. Now pullbacks in $\Psh(\mcc)$ are just pointwise pullbacks in $\pmet$, and we know that pullbacks in $\pmet$ preserve r-images and epis in $\pmet$. Therefore for each $a \in \mcc$, we have that $\psi^\ast \varphi_a = (\psi^\ast i_a) \circ (\psi^\ast e_a)$ is an r-image factorization in $\pmet$ with $\psi^\ast i_a$ a regular mono and $\psi^\ast e_a$ an epi (since all r-image factorizations in $\Psh(\mcc)$ arise as pointwise r-image factorizations in $\pmet$).

The (second half of the) proof of (\ref{pshims}) shows that when we factorize $\psi^\ast \varphi = (\psi^\ast i) \circ (\psi^\ast e)$ such that componentwise $\psi^\ast \varphi_a = (\psi^\ast i_a) \circ (\psi^\ast e_a)$ is an r-image factorization in $\pmet$ with $\psi^\ast i_a$ a regular mono and $\psi^\ast e_a$ an epi, then $\psi^\ast \varphi = (\psi^\ast i) \circ (\psi^\ast e)$ itself is an r-image factorization in $\Psh(\mcc)$.

(\ref{pshregmonoscomp}): Regular monos in $\Psh(\mcc)$ are just natural transformations whose components are isometric embeddings, and clearly these are closed under composition.

\end{proof}

\begin{proposition}
\label{prop:pshgeom}

$\Psh(\mcc)$ is an r-geometric category.

\end{proposition}

\begin{proof}

Let $F \in \Psh(\mcc)$. We show that $\Sub F$ is a small-complete lattice and that pulling back across any $\varphi: G \rightarrow F$ preserves this structure.

Given $\mu_i \in \Sub F$, we have the data $\mu_i (a) \in \Sub \left ( Fa \right )$ for each $a \in \mcc$. We verify that setting $\left ( \bigwedge\limits_i \mu_i \right ) (a) = \bigwedge\limits_i \left ( \mu_i (a) \right )$ defines the meet of the $\mu_i$; the verification that $\left ( \bigvee\limits_i \mu_i \right ) (a) = \bigvee\limits_i \left ( \mu_i (a) \right )$ defines the join of the $\mu_i$ is entirely analogous.

Therefore given $\mu_i \in \Sub F$, we define $\mu \in \Sub F$ by $\mu(a) = \bigwedge\limits_i \left ( \mu_i (a) \right )$. Certainly $\mu(a) \in \Sub Fa$. Given a morphism $f: a \rightarrow b$ in $\mcc$, we certainly have that $Ff: Fb \rightarrow Fa$ restricts to a map $\mu(f): \mu(b) \rightarrow \mu(a)$. We thus have a subobject $\mu \in \Sub F$ which is the pointwise meet of the $\mu_i$.

Now given any $\nu \in \Sub F$ such that $\nu \leq \mu_i$, we need to verify that $\nu \leq \mu$. We have that $\nu (a) \leq \mu_i (a)$ for all $a \in \mcc$, so we have that $\nu(a) \leq \mu(a)$ for all $a \in \mcc$. By Lemma \ref{lem:sameposet} we conclude that $\nu \leq \mu$.

We have shown that meets (and joins, by an analogous argument) in $\Sub F$ are pointwise meets (and pointwise joins). Pullback across $\varphi: F \rightarrow G$ is given by pointwise pullback, which preserves pointwise meets and joins. Therefore pullbacks preserve meets and joins in $\Psh(\mcc)$.

\end{proof}

We now have all the basic structure on $\Psh(\mcc)$ to repeat our setup from $\pmetu$ in defining a metrization. Since each $F \in \Psh(\mcc)$ takes values in $\pmet$, our discussion following Definition \ref{def:metrization} suggests that we should take $E_1 = \{1_{\mathbb{I}}\} \subseteq [ \mathbb{I} , \mathbb{I} ]$ to be our $\mathbb{I}$-moduloid.

\begin{proposition}
\label{prop:pshmetrization}

$\Psh(\mcc)$ has a metrization compatible with $E_1$.

\end{proposition}

\begin{proof}

$\pmet$ already has a metrization compatible with $E_1$. To any $F, G \in \Psh(\mcc)$, we can make the choice of product $F \times G$ given by $(F \times G) a = Fa \times Ga$ for each $a \in \mcc$ where the latter product is the choice of product given by the metrization of $\pmet$.

Now for each $F \in \Psh(\mcc)$, we define $D_F \in [\Iop, \Sub (F \times F) ]$ as follows.

For each $r \in \Iop$, set $D_F (r) \in \Sub (F \times F)$ to be the subobject of $F \times F$ which assigns $D_{Fa} (r)$ to each $a \in \mcc$. For any morphism $f: a \rightarrow b$ of $\mcc$ we have that $Ff: Fb \rightarrow Fa$ is $1$-Lipschitz, so certainly we have that $Ff \times Ff : (Fb \times Fb) \rightarrow (Fa \times Fa)$ restricts to $Ff \times Ff: D_{Fb}(r) \rightarrow D_{Fa}(r)$. Therefore by Lemma \ref{lem:whataresubs} at least each $D_F(r)$ as we have defined it is actually a subobject of $F \times F$.

To see that $r \mapsto D_F(r)$ defines a functor $D_F: \Iop \rightarrow \Sub (F \times F)$, we check that whenever $r \leq s$, we also have $D_F(r) \leq D_F(s)$. But if $r \leq s$ we have that $D_F(r) (a) \leq D_F(s) (a)$ for each $a \in \mcc$, therefore by Lemma \ref{lem:sameposet} we have that $D_F(r) \leq D_F(s)$.

It is a straightforward verification that this construction satisfies all of the conditions of Definition \ref{def:metrization}; the main idea is that all of the structure involved is inherited pointwise from $\pmet$, and thus already satisfies the conditions, which can also all be checked pointwise.

\end{proof}

We have therefore established a notion of ``metric'' on each $F: \Iop \rightarrow \pmet$ with attendant notions of continuity, (continuous) predicates, (continuous) quantification, etc. As a corollary Definitions \ref{def:econtsubobj2} and \ref{def:predclass} make sense in $\Psh(\mcc)$, and also the results of Proposition \ref{prop:metrizationcors} and \ref{prop:freeresults} apply to $\Psh(\mcc)$.

Note that since our $\mathbb{I}$-moduloid $E$ is trivial ($E$ only contains $1_{\mathbb{I}}: \mathbb{I} \rightarrow \mathbb{I}$), we have no need to specify moduli of continuity $\epsilon$; it is always just $1_{\mathbb{I}}$. Thus in our case, for any predicate $R$ on $X$ we have that $R$ is a $1_{\mathbb{I}}$-predicate on $X$, i.e. $R \in \Suboi X$. Since no confusion may result, for tidiness of notation we henceforth write $\Suboi X$ as $\Subi X$.

We now show that $\Psh(\mcc)$ in fact has a predicate classifier, in the sense of Definition \ref{def:predclass}.

\begin{theorem}

$\Psh(\mcc)$ has a predicate classifier.

\end{theorem}

\begin{proof}

To give a predicate classifier, we must specify an object $\Omega \in \Psh(\mcc)$ - and therefore also its metric $D_{\Omega} \in [\Iop, \Sub (\Omega \times \Omega) ]$ - along with a specified $\mct_{\mathbb{I}} \in \Subi \Omega$ such that for any $R \in \Subi F$ there is a unique $\varphi: F \rightarrow \Omega$ such that $R = \varphi^{\ast} \mct_{\mathbb{I}}$.

For each $a \in \mcc$, denote by $\mcs_a$ the poset of sieves on $a$, partially ordered by inclusion (sieves are just sets of morphisms) and regarded as a category. Let us define the set $| \Omega a |$ to be the set of functors $S: \Iop \rightarrow \mcs_a$ satisfying the following property:

For $r, r_i \in \mathbb{I}$ with $\inf\limits_i r_i = r$,
\begin{equation*}
S ( r ) = \bigwedge\limits_i \left ( S(r_i) \right ).
\end{equation*}

Now given $a, b \in \mcc$ and $f: b \rightarrow a$, we define $| \Omega f |: | \Omega a | \rightarrow | \Omega b |$ as the evident pullback action on sieves. That is, given $S \in | \Omega a |$ which is thus a functor $S: \Iop \rightarrow \mcs_a$, we specify the functor $| \Omega f | (S) : \Iop \rightarrow \mcs_b$ by
\begin{equation*}
| \Omega f | (S) (r) = f^{\ast} \left ( S(r) \right ) = \{ h: b^\prime \rightarrow b \mid fh \in S(r) \}.
\end{equation*}

So far we have specified the values, at the $\catsets$ level, of a hypothetical functor $\Omega: \mcc \rightarrow \pmet$. What remains is to put a metric on $| \Omega a |$ for each $a \in \mcc$ so that the set function $| \Omega f |: | \Omega a | \rightarrow | \Omega b |$ for each $f: b \rightarrow a$ is $1$-Lipschitz with respect to these metrics, and then to check that $\Omega$ is a classifier.

We define a (set) function $\nu_a : | \Omega a | \rightarrow [0,1]$ by
\begin{equation*}
\nu_a (S) = \inf \{ r \mid S(r) \text{ is the maximal sieve on $a$} \}
\end{equation*}
and a (set) function $d_a^-: | \Omega a | \times | \Omega a | \rightarrow [0,1]$ by
\begin{equation*}
d_a^- (S_1, S_2) = | \nu_a (S_1) - \nu_a (S_2) |.
\end{equation*}
(Note that, by construction of $| \Omega a |$, $\nu_a(S) \leq r$ iff $S(r)$ is the maximal sieve on $a$.)

We now define $d_a : | \Omega a | \times | \Omega a | \rightarrow [0,1]$ by
\begin{equation*}
d_a (S_1, S_2) = \sup\limits_{f: b \rightarrow a} d_b^- \left ( |\Omega f| (S_1), |\Omega f| (S_2) \right )
\end{equation*}
where the supremum is taken over all morphisms $f: b \rightarrow a$ in $\mcc$ with codomain $a$.

We claim that $d_a$ gives a metric on each $| \Omega a |$. Clearly $d_a^-$ is a metric for each $a \in \mcc$. Reflexivity and symmetry of $d_a$ are also clear; we check the triangle inequality for $d_a$.

Let $S_1, S_2, S_3 \in | \Omega a |$. Let $\delta = d_a ( S_1, S_3 ) = \sup\limits_{f: b \rightarrow a} d_b^- \left ( |\Omega f| (S_1), |\Omega f| (S_3) \right )$. For every $\delta_i < \delta$, there is some $f_i: b_i \rightarrow a$ such that $d_{b_i}^- \left ( |\Omega f_i | (S_1), | \Omega f_i | (S_3) \right ) > \delta_i$. By the triangle inequality for $d_{b_i}^-$, we have that
\begin{equation*}
d_{b_i}^- \left ( |\Omega f_i | (S_1), | \Omega f_i | (S_2) \right ) + d_{b_i}^- \left ( |\Omega f_i | (S_2), | \Omega f_i | (S_3) \right ) \geq d_{b_i}^- \left ( |\Omega f_i | (S_1), | \Omega f_i | (S_3) \right ) > \delta_i.
\end{equation*}
This shows that $d_a ( S_1, S_2) + d_a ( S_2, S_3) \geq \delta = d_a (S_1, S_3)$.

Now for each $a \in \mcc$, we define $\Omega a$ to be the metric space whose underlying set is $| \Omega a |$ and whose metric is $d_a$. It is easy to see that for each $f: b \rightarrow a$, the function $| \Omega f |: | \Omega a | \rightarrow | \Omega b |$ is the underlying set function of a $1$-Lipschitz map $\Omega f: \Omega a \rightarrow \Omega b$. It is clear that the assignment $f \mapsto |\Omega f|$, and therefore also $f \mapsto \Omega f$, is functorial. Thus we have a functor $\Omega: \mcc \rightarrow \pmet$. Proposition \ref{prop:pshmetrization} tells us that the ``pointwise metric'' $d_a$ on each $\Omega a$ gives us $D_{\Omega} \in [\Iop, \Sub (\Omega \times \Omega) ]$.

We now specify $\mct_{\mathbb{I}} \in \Subi \Omega$. We first define $\mct_{\mathbb{I}} \in [\Iop, \Sub \Omega ]$, and then show that $\mct_{\mathbb{I}}$ satisfies the conditions to be a ($1_{\mathbb{I}}$-)predicate on $\Omega$. For each $r \in \mathbb{I}$, we specify (a hypothetical) $\mct_{\mathbb{I}} (r) \in \Sub \Omega$ by the assignment $a \mapsto \nu_a^{-1}( [0,r] ) \subseteq \Omega a$ for each $a \in \mcc$. Given $f: b \rightarrow a$, $\Omega f: \Omega a \rightarrow \Omega b$ takes $\nu_a^{-1}( [0,r] )$ into $\nu_b^{-1}( [0,r] )$ since the pullback along $f: b \rightarrow a$ of the maximal sieve on $a$ is the maximal sieve on $b$, so we indeed have $\mct_{\mathbb{I}} (r) \in \Sub \Omega$. Clearly $\mct_{\mathbb{I}} (r) \leq \mct_{\mathbb{I}} (s)$ whenever $r \leq s$, thus we have defined $\mct_{\mathbb{I}} \in [\Iop, \Sub \Omega]$.

We show that $\mct_{\mathbb{I}}$ as given above satisfies
\begin{enumerate}

\item For $r, r_i \in \mathbb{I}$ with $\inf\limits_i r_i = r$,
\begin{equation*}
\mct_{\mathbb{I}} (r) = \bigwedge\limits_i \left ( \mct_{\mathbb{I}} (r_i) \right )
\end{equation*}\label{mctclosure}

\item For $r, s \in \mathbb{I}$,
\begin{equation*}
\left ( \underline{\pi_1}^{\ast} \mcti(r) \right ) \wedge D_{\Omega} (s) \leq \underline{\pi_2}^{\ast} \mcti(r + s)
\end{equation*}\label{mctcont}

\end{enumerate}
thereby showing that $\mct_{\mathbb{I}} \in \Subi \Omega$.

(\ref{mctclosure}): Let $r, r_i \in \mathbb{I}$ with $\inf\limits_i r_i = r$. For any $a \in \mcc$, we have that $\mcti(r) (a) = \nu_a^{-1}([0,r])$, therefore the set of functors $S: \Iop \rightarrow \mcs_a$ such that $S(r)$ is the maximal sieve on $a$.

Clearly $\mcti(r) (a) = \nu_a^{-1}([0,r]) \leq \bigwedge\limits_i \left ( \mct_{\mathbb{I}} (r_i) (a) \right )$. Now if $S \in \bigwedge\limits_i \left ( \mct_{\mathbb{I}} (r_i) (a) \right )$ then $S(r_i)$ is the maximal sieve on $a$ for all $i$, and therefore $\nu_a (S) \leq r$. Therefore we have $\bigwedge\limits_i \left ( \mct_{\mathbb{I}} (r_i) (a) \right ) \leq \mcti(r)(a)$ for all $a \in \mcc$.

(\ref{mctcont}): It suffices to check this condition pointwise, where it turns into
\begin{equation*}
\left ( \underline{\pi_1}^{\ast} \mcti(r)(a) \right ) \wedge D_{\Omega a} (s) \leq \underline{\pi_2}^{\ast} \mcti(r + s) (a)
\end{equation*}
where $D_{\Omega a}$ is just the object of $[\Iop, \Sub (\Omega a \times \Omega a)]$ naturally given by the metric $d_a$ of $\Omega a$. Thus the above condition translates into
\begin{equation*}
x \in \mcti(r)(a) \text{ and } d_a(x, y) \leq s \Longrightarrow y \in \mcti(r + s)(a).
\end{equation*}
$x \in \mcti(r)(a)$ means that $x$ is a functor $x: \Iop \rightarrow \mcs_a$ such that $x(\inf\limits_i r_i) = \bigwedge\limits_i x(r_i)$ and $x(r)$ is the maximal sieve on $a$. Now
\begin{equation*}
d_a(x,y) \leq s \Longleftrightarrow \sup\limits_{f: b \rightarrow a} \left | \nu_b ( |\Omega f| x) - \nu_b ( |\Omega f| y ) \right | \leq s
\end{equation*}
so in particular, taking $f$ to be the identity $1_a: a \rightarrow a$, we have $\nu_a (y) - \nu_a (x) \leq s$. Therefore $y ( r + s )$ is the maximal sieve on $a$ and $y \in \mcti(r + s) (a)$.

So far we have constructed $\Omega \in \Psh(\mcc)$ and $\mcti \in \Subi \Omega$. We now show that these satisfy the conditions given in Definition \ref{def:predclass} to be a predicate classifier. That is, for any $F \in \Psh(\mcc)$ and any $R \in \Subi F$, there is a unique morphism $\varphi: X \rightarrow \Omega$ such that $\varphi^{\ast} \mcti = R$.

Specifically, this means the following. Any $R \in \Subi F$ is precisely the data assigning a subspace $R(r) (a) \subseteq F a$ to every $r \in \mathbb{I}$ and $a \in \mcc$ so that:

\begin{enumerate}[label=$\circ$, ref=$\circ$]

\item For every $r \leq s$ and $a \in \mcc$, $R(r)(a) \subseteq R(s)(a)$;

\item For every $r \in \mathbb{I}$, and $f: b \rightarrow a$, $Ff: Fa \rightarrow Fb$ restricts to a map $Ff: R(r)(a) \rightarrow R(r)(b)$.

\end{enumerate}

We need to show that there is a unique $\varphi: F \rightarrow \Omega$ for which $\varphi^{\ast} \mcti = R$, which means specifying the data of maps $\varphi_a: Fa \rightarrow \Omega a$ for each $a \in \mcc$ satisfying the following conditions:

\begin{enumerate}[label=$\circ$, ref=$\circ$]

\item For each $f: b \rightarrow a$, $\varphi_b \circ Ff = \Omega f \circ \varphi_a$;

\item For each $r \in \mathbb{I}$ and $a \in \mcc$,
\begin{equation*}
\varphi_a ^{-1} \left ( \nu_a^{-1} ( [0, r] ) \right ) = R (r) (a).
\end{equation*}

\end{enumerate}

There is at least a \emph{set} function $\varphi_a$ for each $a \in \mcc$ satisfying the above conditions: for each $x \in Fa$, we define $\varphi_a (x): \Iop \rightarrow \mcs_a$ to be the functor such that, for all $r \in \mathbb{I}$,
\begin{equation*}
\varphi_a(x)(r) = \{ f: b \rightarrow a \mid Ff (x) \in R(r)(b) \}.
\end{equation*}

Unwinding the definitions, we see that this definition of $\varphi_a$ satisfies
\begin{equation*}
x \in \varphi_a^{-1} \left ( \nu_a^{-1}([0,r]) \right ) \Longleftrightarrow x \in R(r)(a),
\end{equation*}
since the former set is the set of points $x \in Fa$ for which the sieve $\{ f: b \rightarrow a \mid Ff (x) \in R(r)(b) \}$ is the maximal one. Moreover, for any $f: b \rightarrow a$ we have that $\varphi_b ( Ff(x) ): \Iop \rightarrow \mcs_b$ is the functor whose value at $r \in \mathbb{I}$ is the sieve $\{h: b^\prime \rightarrow b \mid Fh ( Ff(x) ) = F (fh) (x) \in R(r)(b^\prime) \}$. But this sieve is exactly $\Omega f ( \varphi_a (x) (r) )$. It is straightforward to verify that that this is the only way to define the set functions $\varphi_a$ which satisfy the above conditions. Thus the proof will be finished once we show that this definition of $\varphi_a$ is $1$-Lipschitz for each $a \in \mcc$.

Given $x, y \in Fa$, we verify that $d_a (\varphi_a (x), \varphi_a (y) ) \leq d_{Fa}(x,y)$.
\begin{equation*}
d_a (\varphi_a (x), \varphi_a (y) ) = \sup\limits_{f: b \rightarrow a} \left | \nu_b ( \Omega f (\varphi_a (x)) ) - \nu_b ( \Omega f (\varphi_a (y)) ) \right |
\end{equation*}
so it suffices to show that for any $f: b \rightarrow a$, we have
\begin{equation*}
\left | \nu_b ( \Omega f (\varphi_a (x)) ) - \nu_b ( \Omega f (\varphi_a (y)) ) \right | \leq d_{Fa} (x, y).
\end{equation*}
By construction, we have that $\varphi_b \circ Ff = \Omega f \circ \varphi_a$, so the above is equivalent to
\begin{equation*}
\left | \nu_b ( \varphi_b ( Ff (x) ) ) - \nu_b ( \varphi_b ( Ff (y) ) ) \right | \leq d_{Fa} (x, y).
\end{equation*}
But we claim that
\begin{equation*}
\left | \nu_b ( \varphi_b ( Ff (x) ) ) - \nu_b ( \varphi_b ( Ff (y) ) ) \right | \leq d_{Fb} (Ff (x), Ff (y) ) = \delta,
\end{equation*}
which we now prove. Without loss of generality, we can assume $r_x = \nu_b ( \varphi_b ( Ff(x) ) ) \leq \nu_b ( \varphi_b ( Ff (y) ) ) = r_y$, so that the left hand side in the above is just $r_y - r_x$. By construction, $r_x = \inf \{ r \mid Ff(x) \in R(r)(b) \}$. Since $R$ is a predicate, we have that $Ff(x) \in R(r_x)(b)$. Again since $R$ is a predicate, we must therefore have that $Ff(y) \in R(r_x + \delta) (b)$, so that $r_y \leq r_x + \delta$ and therefore $r_y - r_x \leq \delta$. Since $\delta \leq d_{Fa}(x,y)$ we are done.

\end{proof}

\bibliography{contsem}
\bibliographystyle{plain}

\end{document}